\documentclass[reqno]{amsart}
\usepackage{color,epsfig}
\usepackage[applemac]{inputenc}
\usepackage[T1]{fontenc}
\usepackage{amsmath,amssymb,amsthm,mathdots}
\usepackage[english]{babel}
\usepackage{graphicx}

\newtheorem{theorem}{Theorem}[section]
\newtheorem{lemma}[theorem]{Lemma}
\newtheorem{corollary}[theorem]{Corollary}
\newtheorem{proposition}[theorem]{Proposition}
\newtheorem*{propbis}{Proposition \ref{non-dense} bis}

\theoremstyle{definition}
\newtheorem{definition}[theorem]{Definition}
\newtheorem*{nota}{Notation}
\newtheorem{ques}[theorem]{Question}

\theoremstyle{remark}
\newtheorem{rema}[theorem]{Remark}
\newtheorem{example}[theorem]{Example}

\DeclareMathOperator{\id}{id}
\DeclareMathOperator{\lo}{{\rm length}}
\DeclareMathOperator{\NS}{{\rm N}^1}
\DeclareMathOperator{\SL}{\rm{SL}}
\DeclareMathOperator{\PSL}{\rm{PSL}}
\DeclareMathOperator{\GL}{\rm{GL}}
\DeclareMathOperator{\vol}{\rm vol}
\DeclareMathOperator{\mvol}{\rm mvol}
\DeclareMathOperator{\rg}{\rm rank}
\DeclareMathOperator{\Aut}{\rm Aut}
\DeclareMathOperator{\Diff}{\rm Diff}
\DeclareMathOperator{\Homeo}{\rm Homeo}
\DeclareMathOperator{\Isom}{\rm O}
\DeclareMathOperator{\SO}{\rm SO}
\DeclareMathOperator{\h}{{\rm h_{\it top}}}
\DeclareMathOperator{\liap}{\chi_{\it top}}
\DeclareMathOperator{\M}{\rm M}
\DeclareMathOperator{\Amp}{\rm Amp}
\DeclareMathOperator{\Nef}{\rm Nef}
\DeclareMathOperator{\Pos}{\rm Pos}

\renewcommand\epsilon{\varepsilon}
\renewcommand\Im{\mathfrak{Im}}
\renewcommand\Re{\mathfrak{Re}}
\newcommand\Z{\mathbf{Z}}
\newcommand\R{\mathbf{R}}
\newcommand\C{\mathbf{C}}
\newcommand\Q{\mathbf{Q}}
\newcommand\N{\mathbf{N}}
\renewcommand\P{\mathbf{P}}
\newcommand\K{\mathbf{K}}
\renewcommand\b{\overline}
\renewcommand{\O}{\mathcal O}

\begin{document}

\title[Real versus complex volumes on real algebraic surfaces]{Real versus complex volumes\\on real algebraic surfaces}

\author{Arnaud Moncet}
\address{Arnaud Moncet\\
Universit\'e de Rennes 1\\
IRMAR\\
Campus de Beaulieu\\
b\^at. 22-23\\
263 avenue du G\'en\'eral Leclerc\\
CS 74205\\
35042 Rennes cedex
}
\email{arnaud.moncet@univ-rennes1.fr}
\date{July 2011}

\maketitle

\begin{abstract}
Let $X$ be a real algebraic surface. The comparison between the volume of $D(\R)$ and $D(\C)$ for ample divisors $D$ brings us to define the concordance $\alpha(X)$, which is a number between $0$ and $1$. This number equals $1$ when the Picard number $\rho(X_\R)$ is $1$, and for some surfaces with a ``quite simple'' nef cone, e.g. Del Pezzo surfaces. For abelian surfaces, $\alpha(X)$ is $1/2$ or $1$, depending on the existence or not of positive entropy automorphisms on $X$. In the general case, the existence of such an automorphism gives an upper bound for~$\alpha(X)$, namely the ratio of entropies $\h(f{|X(\R)})/\h(f{|X(\C)})$. Moreover $\alpha(X)$ is equal to this ratio when the Picard number is $2$. An interesting consequence of the inequality is the nondensity of $\Aut(X_\R)$ in $\Diff(X(\R))$ as soon as $\alpha(X)>0$. Finally we show, thanks to this upper bound, that there exist K3 surfaces with arbitrary small concordance,  considering a deformation of a singular surface of tridegree $(2,2,2)$ in $\P^1\times\P^1\times\P^1$.
\end{abstract}



\section{Introduction}

Let $X$ be a real projective variety with a fixed Riemannian metric. The goal of this paper is  to compare volumes of real subvarieties $Y(\R)$ and of their complexifications $Y(\C)$. As will be seen, this is closely related to the question of comparing real and complex dynamics of automorphisms of $X_\R$.


\subsection{Projective space}

Consider the projective space $X=\P^d_\R$, equipped with the Fubini--Study metric. Let $Y$ be a real subvariety of $\P^d_\R$ of dimension $k$. By Wirtinger's formula (see \cite[p. 31]{griffiths-harris}), the volume of $Y(\C)$ satisfies 
\begin{equation}
\vol_\C(Y)=\deg(Y)\vol_\C(\P^k).
\end{equation}
For the volume of $Y(\R)$, the Cauchy--Crofton formula enables us to show that
\begin{equation}
\vol_\R(Y)\leq\deg(Y)\vol_\R(\P^k)
\end{equation}
and to characterize the case of equality (see Appendix \ref{annexe-crofton}). This gives the following proposition.

\begin{proposition}\label{prop-proj}
Let $Y$ be a real $k$-dimensional algebraic subvariety of the projective space $\P^d_\R$. With respect to the Fubini--Study metric, we have
\begin{equation}\label{ineg-proj}
\frac{\vol_\R(Y)}{\vol_\R(\P^k)}\leq \frac{\vol_\C(Y)}{\vol_\C(\P^k)}.
\end{equation}
Furthermore, equality is achieved if and only if $Y$ is the union of $\deg(Y)$ real projective subspaces.
\end{proposition}

As a consequence, if $\mathcal{V}_k(\delta)$ denotes the set of real subvarieties of $\P^d_\R$ of dimension~$k$ and degree $\delta$, then for any $Y_0\in\mathcal{V}_k(\delta)$ we have
\begin{equation}\label{conc-esp-proj}
\max_{Y\in\mathcal{V}_k(\delta)}\vol_\R(Y)=\delta\vol_\R(\P^k)=\frac{\vol_\R(\P^k)}{\vol_\C(\P^k)}\vol_\C(Y_0).
\end{equation}


\subsection{General case}

Now $X$ is an arbitrary $d$-dimensional real algebraic variety. We assume that it is \emph{projective}, \emph{smooth}, \emph{irreducible} and that the real locus $X(\R)$ is \emph{not empty}. Let $Y$ be a $k$-dimensional real algebraic subvariety of $X$. Denote the volume of $Y(\R)$ by $\vol_\R(Y)$, and that of $Y(\C)$ by $\vol_\C(Y)$, both with respect to a fixed Riemannian metric on $X(\C)$.

\begin{nota}
Let $\mathcal{V}(Y)$ be the family of real algebraic subvarieties $Z$ such that $Y(\C)$ and $Z(\C)$ have the same homology class in ${\rm H}_{2k}(X(\C);\Z)$. Then for $\K=\R$ or $\C$ we set
\begin{equation}
\mvol_\K(Y)=\max_{Z\in\mathcal{V}(Y)}\vol_\K(Z).
\end{equation}
\end{nota}

When the Riemannian metric comes from the Fubini--Study metric on some projective space $\P^n$ in which $X$ is embedded, we get Inequality (\ref{ineg-proj}). Since two Riemannian metrics are comparable (by compactness of $X(\C)$), we obtain the following proposition.

\begin{proposition}\label{mvolr-leq}
Let $X$ be a real algebraic variety, equipped with an arbitrary Riemannian metric. For any $k\in\N^*$ there exists a constant $C_k>0$ depending on the choice of the metric, such that
\begin{equation}\label{ineg-gen}
\mvol_\R(Y)\leq C_k\mvol_\C(Y)
\end{equation}
for all $k$-dimensional subvarieties $Y$ of $X$.
\end{proposition}

Now we would like to know for which nonnegative exponents $\alpha$ we can write inequalities such as $\mvol_\R(Y)\geq C_k\mvol_\C(Y)^\alpha$, with $C_k$ independent of $Y$. We restrict ourselves to codimension $1$ subvarieties, that is, effective divisors. The notion of homology class in ${\rm H}_{2d-2}(X(\C);\Z)$ is dual to that of (first) Chern class in~${\rm H}^2(X(\C);\Z)$, which is preferred in what follows.

\begin{definition}\label{defi-concordance}
Let $\mathcal{A}(X)$ be the set of nonnegative exponents $\alpha$ for which there exist $C>0$ and $q\in\N^*$ such that
\begin{equation}\label{eq-alpha}
\mvol_\R(D)\geq C\mvol_\C(D)^\alpha
\end{equation}
for all real ample divisors $D$ whose Chern classes are $q$-divisible. The upper bound of $\mathcal{A}(X)$ is the \emph{concordance} of $X$, and is denoted by $\alpha(X)$. We say the concordance is \emph{achieved} when $\alpha(X)$ is contained in $\mathcal{A}(X)$.
\end{definition}

The set $\mathcal{A}(X)$, and thus the concordance $\alpha(X)$, only depend on $X$, and not on the choice of a particular metric. All metrics will be K\"ahler metrics, so that the number $\vol_\C(D)$ only depends on the Chern class of $D$; thus we write $\vol_\C(D)$ instead of $\mvol_\C(D)$.

As will be seen in Section \ref{section-concordance}, the concordance $\alpha(X)$ can only take values between~$0$ and $1$, and the set $\mathcal{A}(X)$ is an interval of the form $[0,\alpha(X)]$ or $[0,\alpha(X))$, whether the concordance is achieved or not.


\subsection{Examples}

Equation (\ref{conc-esp-proj}) implies that the concordance of the projective space is $1$. More generally we prove in Section \ref{section-alpha=1} that $\alpha(X)=1$ as soon as the closed convex cone $\Nef(X_\R)$ of real nef $\R$-divisors is generated by finitely many divisors~$D_j$ with $\mvol_\R(D_j)>0$. This is the case when the real Picard number $\rho(X_\R)$ is $1$. As a special case, the concordance of a curve is always $1$. Thus, nontrivial cases (those with $\alpha(X)<1$) can only occur when both the dimension and the Picard number are at least $2$.

In this paper, we focus on the case of surfaces, which already include many interesting examples. Among them, tori are the simplest surfaces for which the concordance is not always $1$ (cf \textsection \ref{tores}):

\begin{theorem}\label{thm-intro-ab}
Let $X$ be a real abelian surface. The real Picard number $\rho(X_\R)$ is equal to $1$, $2$ or $3$, and we have the following values for concordance:
\begin{itemize}
\item[(1)] If $\rho(X_\R)=1$, then $\alpha(X)=1$.
\item[(2)] If $\rho(X_\R)=2$, then $\alpha(X)=1$ or $1/2$, depending on the existence or not of real elliptic fibrations on $X$.
\item[(3)] If $\rho(X_\R)=3$, then $\alpha(X)=1/2$.
\end{itemize}
In all cases the concordance is achieved.
\end{theorem}

In Section \ref{deformation} we show that there exist surfaces with arbitrary small concordance. More precisely, we prove the following result.

\begin{theorem}\label{thm-intro-k3}
There is a family $(X^t)_{t\in (0,1]}$ of real K3 surfaces embedded in $(\P^1)^3$ such that
\begin{equation}
\lim_{t\to 0}\alpha(X^t)=0.
\end{equation}
\end{theorem}


\subsection{Dynamics of automorphisms}\label{intro-dyn}

Let $X$ be a real algebraic surface. We denote by $\Aut(X_\R)$ the group of (real) automorphisms on $X$, that is, biholomorphic maps~$f:X(\C)\to X(\C)$ that commute with the antiholomorphic involution $\sigma$. For~$\K=\R$ or $\C$ the induced self-map on $X(\K)$ is denoted by $f_\K$.

The dynamics of automorphisms on complex surfaces has been broadly studied in the last decades (one may refer to the references given in the surveys \cite{cantat-panorama} and~\cite{bedford-survey}). Let us remember a few facts:
\begin{itemize}
\item[(1)]
The entropy $\h(f_\C)$ is entirely expressed in terms of the action on the cohomology, according to \cite{gromov-entropie} and \cite{yomdin}; namely, it is equal to the logarithm of the spectral radius (called the \emph{spectral logradius} in what follows) of the induced map $f^*$ on ${\rm H}^2(X(\C);\R)$.
\item[(2)]
Automorphisms that have positive entropy, also called \emph{hyperbolic type} automorphisms, can only occur on tori, K3 surfaces, Enriques surfaces, and (nonminimal) rational surfaces, or on blow-ups of such surfaces at periodic orbits \cite{cantat-cras}. Moreover, examples are known on each of these types of surfaces.
\item[(3)]
For hyperbolic type automorphisms we have $\h(f_\C)\geq\log(\lambda_{10})$ \cite{mcmullen-rat}, ${\lambda_{10}\simeq 1,17628081}$ being the Lehmer number. Moreover, this bound is achieved on some rational surfaces \cite{bedford-kim-degree} \cite{mcmullen-rat} and on some K3 surfaces (C. McMullen gives a nonprojective example in \cite{mcmullen-k3-ent}, and announces that there also exists a projective example).
\end{itemize}

On the other hand the dynamics on $X(\R)$ is not as well understood, for we do not dispose of equivalent tools to study it. For instance the entropy $\h(f_\R)$ cannot be deduced from the action on cohomology; it is bounded from below by the spectral radius of $f_\R^*$ on ${\rm H}^1(X(\R);\R)$ \cite{manning}, and from above by $\h(f_\C)$, but may vary within this interval. In particular, we see that, for hyperbolic type automorphisms, the ratio $\h(f_\R)/\h(f_\C)$ is a number between $0$ and $1$ (for tori it always equals~$1/2$: cf Proposition \ref{alpha-leq-1/2}). As proved by Bedford and Kim in \cite{bedford-kim-maxent}, this ratio may happen to be equal to $1$ for some rational\footnote{Throughout the text, rational means rational over $\C$.} surfaces.

\begin{ques}
Is there an example of a real hyperbolic type automorphism on a~K3 or Enriques surface $X$ for which $\h(f_\R)=\h(f_\C)$ ?
\end{ques}

\begin{ques}
Is there an example of a real hyperbolic type automorphism on a surface $X$ for which ${\h(f_\R)=0}$ ?
\end{ques}

\begin{rema}
In \cite{bedford-kim-maxent} the authors prove the maximality of entropy using only homology of real algebraic curves. In this text, I rather use their volumes, which provide a finer measure than their homology classes.
\end{rema}

In Section \ref{section-entropie} we use a theorem due to Yomdin \cite{yomdin} in order to highlight a link between concordance and this ratio of entropies (which is used to prove Theorems~\ref{thm-intro-ab} and \ref{thm-intro-k3}).

\begin{theorem}\label{alpha-entropie}
Let $X$ be a real algebraic surface. Assume that there exists a real hyperbolic automorphism $f$ on $X$. Then
\begin{equation}\label{minoration-alpha}
\alpha(X)\leq\frac{\h(f_\R)}{\h(f_\C)}.
\end{equation}
Moreover, this inequality becomes an equality when $\rho(X_\R)=2$.
\end{theorem}

\begin{corollary}\label{lehmer}
Let $f$ be a real automorphism of a real algebraic surface $X$. If~$\h(f_\R)>0$, then
\begin{equation}
\h(f_\R)\geq\lambda_{10}\,\alpha(X),
\end{equation}
where $\lambda_{10}$ denotes the Lehmer number, that is, the largest root of the polynomial $x^{10}+x^9-x^7-x^6-x^5-x^4-x^3+x+1$.
\end{corollary}

When $\alpha(X)>0$, these results enables us to show nondensity and discreteness results for $\Aut(X_\R)$ in the group of diffeomorphisms of $X(\R)$, as well as in some of its subgroups (\textsection \ref{section-non-dense}).



\subsection*{Acknowledgements}

The work presented in this paper is part of my PhD thesis. I am especially grateful to my thesis advisor Serge Cantat who suggested the topic, encouraged my progress, and patiently read and corrected this text many times during its preparation. I also thank the referee for giving this paper an exceptionally prompt and thorough reading, and for his helpful and positive comments. I would also like to thank S\'ebastien Gou\"ezel, Yutaka Ishii, Fr\'ed\'eric Mangolte and Anton Zorich for many useful discussions. Finally I would like to thank Laura DeMarco for having invited me to the Workshop on Dynamics at the University of Illinois at Chicago in may 2010, as well as all the organizers of this conference.



\section{First Properties of Concordance}


\subsection{Conventions and notations}

In what follows, $X$ denotes the ambient real algebraic variety, and $d$ its dimension. Moreover, $X$ is always supposed to be \emph{projective}, \emph{smooth}, \emph{irreducible} and with \emph{nonempty} real locus. The set $X(\R)$ is then a real analytic $d$-dimensional manifold, with a finite number of connected components. In contrast, we make no particular assumption for subvarieties $Y$ of~$X$. The antiholomorphic involution that defines the real structure on $X$ is denoted by~$\sigma_X$, or simply by $\sigma$ when no confusion is possible.

The cohomology groups ${\rm H}^k(X(\C);\Z)$ are implicitly taken  \emph{modulo torsion}, so that we can consider them as lattices in ${\rm H}^k(X(\C);\R)$.

The {\it complex N\'eron--Severi group} of $X$, denoted by $\NS(X_\C;\Z)$, is the subgroup of~${\rm H}^2(X(\C);\Z)$ whose elements are Chern classes of divisors on $X(\C)$. We denote by $[D]$ the (first) Chern class of a divisor $D$. By the Lefschetz theorem on $(1,1)$-classes (see \cite[p. 163]{griffiths-harris}), we have
\begin{equation}
\NS(X_\C;\Z)={\rm H}^{1,1}(X(\C);\R)\cap{\rm H}^2(X(\C);\Z).
\end{equation}

The {\it real N\'eron--Severi group} of $X$, denoted by $\NS(X_\R;\Z)$, is the subgroup of~$\NS(X_\C;\Z)$ whose elements are classes of real divisors. Recall that
\begin{equation}\label{pb-sign}
[\sigma(D)]=-\sigma^*[D]
\end{equation}
for any complex divisor $D$ (see \cite[\textsection I.4]{silhol}), where $\sigma^*$ denotes the involution on ${\rm H}^2(X(\C);\Z)$ induced by the complex conjugation $\sigma$. Hence
\begin{equation}
\NS(X_\R;\Z)=\{\theta\in\NS(X_\C;\Z)\,|\,\sigma^*\theta=-\theta\},
\end{equation}

Both $\NS(X_\C;\Z)$ and $\NS(X_\R;\Z)$ are free abelian groups of finite rank. Their respective ranks are the \emph{complex} and \emph{real Picard numbers} of $X$, denoted by $\rho(X_\C)$ and $\rho(X_\R)$. For $\K=\R$ or $\C$ we denote by $\NS(X_\K;\R)$ the subspace of ${\rm H}^{1,1}(X(\C);\R)$ spanned by $\NS(X_\K;\Z)$; it has dimension $\rho(X_\K)$.

When $X$ is a surface, the intersection form gives rise to a nondegenerate quadratic form on ${\rm H}^2(X(\C);\R)$, with integral values on ${\rm H}^2(X(\C);\Z)$. By the Hodge index theorem, its signature on the subspace $\NS(X_{\K};\R)$ is $(1,\rho(X_\K)-1)$. Consequently, the positive cone for the intersection form has two connected components, one of which contains classes of ample divisors. This component is an open convex cone in $\NS(X_{\K};\R)$, denoted by $\Pos(X_\K)$. Other convex cones in $\NS(X_\K;\R)$ have their own interest and are used throughout this text, like the ample cone $\Amp(X_\K)$, the nef cone $\Nef(X_\K)$, which is its closure, and the cone of curves $\b{\rm NE}(X_\K)$ (for surfaces it is the same as the pseudo-effective cone), which is the dual of the last one. For all these notions, we refer to  \cite{lazarsfeld}.


\subsection{Positivity of volumes} 

From now on we fix a K\"ahler metric on the complex manifold $X(\C)$. Its K\"ahler form is denoted by $\kappa$.


\subsubsection{Complex volumes}

Let $D$ be an effective divisor on $X$. The volume of $D(\C)$ only depends on the Chern class $[D]$. More precisely,
\begin{equation}\label{formule-vol}
\vol_\C(D)=\frac{1}{(d-1)!}[\kappa^{d-1}]\cdot[D] > 0.
\end{equation}

\begin{proposition}\label{min-unif-volc}
There exists a positive constant $K$ such that
\begin{equation}
\vol_\C(D)\geq K
\end{equation}
 for all effective divisors $D\neq 0$.
\end{proposition}

\begin{proof}
As all Riemannian metrics are equivalent, it is enough to show the inequality when the metric is the Fubini--Study metric on $\P^n\supset X$. In this case the volume of~$D(\C)$ is  proportional to the degree of $D$ as a subvariety of $\P^n$, which is a positive integer. Thus we get the lower bound with $K=\vol_\C(\P^{d-1})>0$.
\end{proof}


\subsubsection{Real volumes}\label{pos-mvolr}

Let $D$ be a real effective divisor on $X$. Although the volume of~$D(\C)$ is always positive, it may happen that $\vol_\R(D)=0$ for some divisors~$D$. For instance, on $X=\P^d_\R$, for any even degree $\delta$ the divisor $D_\delta$ given by the equation $\sum_{j=0}^dZ_j^\delta=0$ has an empty real locus, hence $\vol_\R(D_\delta)=0$. Yet this divisor is numerically (and even linearly) equivalent to $D'_\delta$ given by $\sum_{j=1}^dZ_j^\delta=Z_0^\delta$, and we have $\vol_\R(D'_\delta)>0$.

Thus what is important is not the positivity of $\vol_\R(D)$, but that of $\mvol_\R(D)$. Remember that $\mvol_\R(D)=\max\{\vol_\R(D')\,|\,D'\in\mathcal{V}(D)\}$, where $\mathcal{V}(D)$ is the set of real effective divisors (numerically) equivalent to $D$. Since 
\begin{equation}
\mathcal{V}(D_1+D_2)\supset\mathcal{V}(D_1)+\mathcal{V}(D_2),
\end{equation}
the function $\mvol_\R$ is superadditive on the set of real effective divisors. In particular, for all $k\in\N^*$,
\begin{equation}
\mvol_\R(kD)\geq k\mvol_\R(D).
\end{equation}

\begin{proposition}\label{pinceau}
Let $D$ be a real effective divisor such that the linear system $|D|$ contains a pencil, that is, $h^0(X,\O_X(D))\geq 2$. Then $\mvol_\R(D)>0$.
\end{proposition}

\begin{proof}
Let $(D_\lambda)_{\lambda\in\P^1(\C)}$ be a real pencil in $|D|$ (in this context, real means $D_{\b\lambda}=\sigma(D_\lambda)$). By Bertini's theorem \cite[p.137]{griffiths-harris} there is a finite set $S\subset\P^1(\C)$ such that, for all $\lambda\notin S$, the subvariety $D_\lambda(\C)$ is smooth away from the base locus~$B$ of the pencil $(D_\lambda)_{\lambda\in\P^1(\C)}$. Let $P\in X(\R)\backslash \left(B\cup\bigcup_{\lambda\in S} D_\lambda\right)$. Then there exists $\lambda$ in~$\P^1(\R)\backslash S$ such that the point $P$ is on (the support of) the divisor $D_\lambda$. As~$D_\lambda$ is smooth at $P$, the real locus $D_\lambda(\R)$ contains an arc around $P$, and thus  $\mvol_\R(D)\geq\vol_\R(D_\lambda)>0$.
\end{proof}

This proposition applies, for instance, when $D$ is very ample. In contrast, it may happen that $\mvol_\R(D)=0$ for some effective divisors that are not ample, as shown in the two following examples. It is for this reason that we restrict ourselves to ample divisors in the definition of concordance.

\begin{example}\label{blowup}
Let $X$ be the variety obtained by blowing up $\P^d_\R$ at two (distinct) complex conjugate points, and let $E$ be the exceptional fiber of the blow-up. For any $k\in\N^*$, we have $\mathcal{V}(kE)=\{kE\}$, thus $\mvol_\R(kE)=0$, for $E(\R)$ is empty. Nevertheless, observe that $[E]$ is not in the closure of the cone $\Pos(X_\R)$, since its self-intersection is negative.
\end{example}

\begin{example}\label{quartique}
Let $C$ be a real smooth quartic in $\P^2_{\R}$ such that $C(\R)$ is empty (for instance, the one given by $Z_0^4+Z_1^4+Z_2^4=0$). Take $8$ pairs of complex conjugate points $(P_i,\b P_i)$ on $C$ in such a way that the linear class of $\sum_i(P_i+\b P_i) -\O_{\P^2}(4)_{|C}$ is not a torsion point of ${\rm Pic}^0(C)$. Let $\pi:X\to\P^2$ be the blow-up morphism above these $16$ points (defined over $\R$) and let $C'$ be the strict transform of $C$ in~$X$. For all divisors $D$ in $\mathcal{V}(kC')$ the curve $\pi_*D$ has degree $4k$ and passes through the~$16$ blown-up points with multiplicity at least $k$. Then the choice of the points~$P_i$ implies that $\pi_*D=kC$, hence $D=kC'$. So we see that $\mvol_\R(kC')=0$ for all~$k\in\N^*$. Here~$C'$ is a nef divisor that is not ample as an irreducible divisor with self-intersection~$0$ (see \cite[\textsection 1.4]{lazarsfeld}).
\end{example}


\subsection{The interval $\mathcal{A}(X)$}\label{section-concordance}

Remember that (Definition \ref{defi-concordance}) $\mathcal{A}(X)$ is the set of exponents $\alpha\geq 0$ for which there exist $C>0$ and $q\in\N^*$ such that, for all real ample divisors $D$ with $[D]$ $q$-divisible, we have $\mvol_\R(D)\geq C\vol_\C(D)^\alpha$. This set depends only on $X$, and not on the choice of the metric.

\begin{lemma}
Assume that $\alpha\in\mathcal{A}(X)$. Then $\beta\in\mathcal{A}(X)$ for all $0\leq\beta<\alpha$.
\end{lemma}

\begin{proof}
By Proposition \ref{min-unif-volc}, there exists $K>0$ such that $\mvol_\C(D)\geq K$ for all real effective divisors $D$. When $[D]$ is $q$-divisible, we then have
\begin{equation}
\mvol_\R(D)\geq C\vol_\C(D)^\alpha \geq CK^{\alpha-\beta}\vol_\C(D)^{\beta},
\end{equation}
and so $\beta$ is in $\mathcal{A}(X)$ too.
\end{proof}

As a consequence, $\mathcal{A}(X)$ is an interval of the form $[0,\alpha(X)]$ or $[0,\alpha(X))$. By definition, $\alpha(X)$ is the concordance of $X$.

\begin{lemma}
Let $X$ be  a real algebraic variety (with $X(\R)\neq\emptyset$). The concordance~$\alpha(X)$ is in the interval $[0,1]$.
\end{lemma}

\begin{proof}
Let $\alpha\in\mathcal{A}(X)$. By Proposition \ref{mvolr-leq}, there exists a positive $C'>0$ such that~$\mvol_\R(D)\leq C'\vol_\C(D)$ for all real ample divisors $D$. When $[D]$ is also $q$-divisible, we obtain, for all $k\in\N^*$,
\begin{equation}
C\vol_\C(kD)^\alpha\leq\mvol_\R(kD)\leq C'\vol_\C(kD).
\end{equation}
If $\alpha>1$, this contradicts $\lim_{k\to+\infty}\vol_\C(kD)=+\infty$.
\end{proof}


\subsection{Examples of varieties with concordance 1}\label{section-alpha=1}

We have seen in the introduction that $\alpha(\P^d_\R)=1$. More generally, the concordance is $1$ when the structure of the nef cone is ``simple''.

\begin{proposition}\label{conenef}
Let $X$ be a real algebraic variety. Assume that the cone $\Nef(X_\R)$ is polyhedral, with extremal rays spanned by classes $[D_j]$ such that $\mvol_\R(D_j)>0$. Then the concordance $\alpha(X)$ is $1$, and it is achieved.
\end{proposition}

\begin{proof}
Define $C=\min_j\left(\mvol_\R(D_j)/\vol_\C(D_j)\right)>0$. The classes $[D_j]$ span a finite index subgroup of $\NS(X_\R;\Z)$. Denote by $q$ this index. Since $\Amp(X_\R)\subset\Nef(X_\R)$, every real ample divisor $D$ with $[D]$ $q$-divisible is equivalent to a divisor of the form~$\sum_j k_jD_j$, where the $k_j$'s are nonnegative integers. Hence
\begin{equation}
\mvol_\R(D) \geq \sum_jk_j\mvol_\R(D_j)
\geq C\sum_jk_j\vol_\C(D_j)
= C\vol_\C(D).
\end{equation}
We can then conclude that $1$ is contained in $\mathcal{A}(X)$.
\end{proof}

\begin{corollary}\label{rho=1}
All real algebraic varieties $X$ with $\rho(X_\R)=1$ have concordance $1$, and this one is achieved.
\end{corollary}

\begin{corollary}\label{DP}
Let $X$ be a real Del Pezzo surface. The concordance of $X$ is $1$, and it is achieved.
\end{corollary}

\begin{proof}[Proof of Corollary \ref{DP}]
By definition, a surface is Del Pezzo when its anticanonical divisor $-K_X$ is ample. The cone of curves $\b{\rm NE}(X_\R)$ is then rational polyhedral by the cone theorem (see \cite[1.5.33, 1.5.34]{lazarsfeld}). Thus its dual cone $\Nef(X_\R)$ is also rational polyhedral. 

\begin{lemma}\label{nef}
Let $D$ be a nef divisor on a Del Pezzo surface $X$, which is not numerically trivial. Then the linear system $|D|$ contains a pencil.
\end{lemma}

\begin{proof}
The proof is a simple application of the Riemann--Roch formula:
\begin{equation}
h^0(X,\O_X(D))-h^1(X,\O_X(D))+h^2(X,\O_X(D))=\chi(\O_X)+\frac{1}{2}(-K_X\cdot D+D^2).
\end{equation}

As $-K_X$ is ample and $D$ is nef, $-K_X\cdot D>0$ and $D^2\geq 0$. By Serre duality, we get $h^2(X,\O_X(D))=h^0(X,\O_X(K_X-D))=0$, because $D\cdot(K_X-D)<0$ with~$D$ nef. We conclude that $h^0(X,\O_X(D))>\chi(\O_X)=1$ (the last equality follows from the rationality of $X$).
\end{proof}

Consequently, we see, by Proposition \ref{pinceau}, that the extremal rays of $\Nef(X_\R)$ are spanned by classes $[D_j]$ with $\mvol_\R(D_j)>0$, and thus we can apply Proposition~\ref{conenef} to get the desired result.
\end{proof}



\section{Concordance and Entropy of Automorphisms}\label{section-entropie}

From now on $X$ is a real algebraic surface equipped with a K\"ahler metric whose K\"ahler form is denoted by $\kappa$.

For any differentiable dynamical system $g:M\to M$ on a compact Riemannian manifold, let $\h(g)$ denote the \emph{topological entropy}, and $\liap(g)$ the \emph{topological Liapunov exponent}, that is,
\begin{equation}
\liap(g)=\lim_{n\rightarrow+\infty}\frac{1}{n}\log\|{\rm D}g^n\|_\infty,
\end{equation}
where the notation $\|{\rm D}g\|_\infty$ stands for $\max_{x\in M}\|{\rm D}g(x)\|$, the norm being taken with respect to the Riemannian metric. Note that the number $\liap(g)$ does not depend on the choice of the Riemannian metric.

When $f\in\Aut(X_\R)$ is a real automorphism of $X$, we are going to look at both differentiable dynamical systems $f_\C:X(\C)\to X(\C)$ and $f_\R:X(\R)\to X(\R)$.

We denote by $f_*$ the inverse of the map $f^*$ induced by $f$ on ${\rm H}^2(X(\C);\Z)$, so that the operation $f\to f_*$ is covariant. The linear map $f_*$ is an isometry for the intersection form and preserves the Hodge structure: we say it is a \emph{Hodge isometry}. Furthermore it is also compatible with the direct image of divisors $D$, which means that $f_*[D]=[f_*D]$. Hence $f_*$ preserves the subgroups $\NS(X_\K;\Z)$ for $\K=\R$ or $\C$. We still denote by $f_*$ the restriction of $f_*$ to all the subgroups or subspaces (when extended by $\R$ or $\C$) that are preserved.

The spectral radius of $f_*$ (\emph{a priori} on ${\rm H}^2(X(\C);\R)$) is denoted by $\lambda(f)$. By the theorem of Gromov and Yomdin recalled in the introduction, we have
\begin{equation}
\h(f_\C)=\log(\lambda(f)).
\end{equation}
This spectral radius is actually achieved on the subspace $\NS(X_\R;\R)$ (cf Remark~\ref{rayon-n1}).


\subsection{Complex volume of the iterates of a divisor}

\begin{theorem}\label{iteres-mvolc}
Let $f$ be an automorphism of a complex algebraic surface $X$. For all ample divisors $D$ we have
\begin{equation}\label{ent-c}
\lim_{n\to+\infty}\frac{1}{n}\log\left(\vol_\C(f_*^nD)\right)=\h(f_\C)=\log(\lambda(f)).
\end{equation}
\end{theorem}

\begin{proof}
Wirtinger's equality gives (cf (\ref{formule-vol})) 
\begin{equation}
\vol_\C(f_*^nD)=f_*^n[D]\cdot[\kappa].
\end{equation}

If $\lambda(f) = 1$, the sequence $\left(\|f_*^n[D]\|\right)_{n\in\N}$ has at most a polynomial growth (actually it is at most quadratic \cite{gizatullin}), as well as $\left(\vol_\C(f_*^nD)\right)_{n\in\N}$. Hence
\begin{equation}
\lim_{n\to+\infty}\frac{1}{n}\log\left(\vol_\C(f_*^nD)\right)=0=\log(\lambda(f)).
\end{equation}

If $\lambda(f)>1$, since $[D]$ is in the ample cone, the sequence $\left(\frac{f_*^n[D]}{\lambda(f)^n}\right)_{n\in\N}$ converges to the class $\theta$ of a positive closed current, by \cite{cantat-k3}. In particular
\begin{equation}\label{volc-borne}
\lim_{n\to+\infty}\frac{\vol_\C(f_*^nD)}{\lambda(f)^n}=\theta\cdot[\kappa]>0,
\end{equation}
and then $\lim_{n\to+\infty}\frac{1}{n}\log\left(\vol_\C(f_*^nD)\right)=\log(\lambda(f))$.
\end{proof}

\begin{rema}\label{rayon-n1}
If, moreover, the surface $X$, the automorphism $f$, and the divisor $D$ are defined over $\R$, then the class $\theta$ is in  $\NS(X_\R;\R)$ (as a limit of classes that are in this closed subspace), and satisfies $f_*\theta=\lambda(f)\theta$. Thus $\lambda(f)$ is an eigenvalue of~${f_*}$ restricted to $\NS(X_\R;\R)$.
\end{rema}

\begin{rema}
For varieties that have arbitrary dimension $d$, the formula
\begin{equation}
\lim_{n\to+\infty}\frac{1}{n}\log\left(\vol_\C(f_*^nD)\right)=\log(\lambda(f))
\end{equation}
still holds. But this is not necessarily equal to the entropy, which is the spectral logradius\footnote{Recall that \emph{spectral logradius} stands for the logarithm of the spectral radius.} on the whole cohomology, \emph{a priori} distinct from the spectral logradius~$\log(\lambda(f))$ on ${\rm H}^2(X(\C);\R)$.
\end{rema}


\subsection{An upper bound for real volume of the iterates of a divisor}

\begin{theorem}\label{iteres-mvolr}
Let $f$ be a real automorphism of a real algebraic surface $X$. For all ample real divisors $D$ we have
\begin{equation}\label{ent-r}
\limsup_{n\to+\infty}\frac{1}{n}\log\left(\mvol_\R(f_*^nD)\right)\leq\h(f_\R).
\end{equation}
\end{theorem}

The proof of this result relies on \cite[Theorem 1.4]{yomdin}, which gives a lower bound for entropy in terms of volume growth. It is here stated in the particular case of dimension $1$ submanifolds.

\begin{theorem}[Yomdin]\label{yomdin}
Let $M$ be a compact Riemannian manifold, $g:M\rightarrow M$ be a differentiable map and $\gamma:[0,1]\to M$ be an arc, each of class $C^r$, with $r\geq 1$. Then \begin{equation}
\limsup_{n\rightarrow+\infty}\frac{1}{n}\log\left(\lo(g^n\circ\gamma)\right)\leq\h(g)+\frac{2}{r}\liap(g).
\end{equation}
In particular when the regularity is $C^\infty$, then
\begin{equation}
\limsup_{n\rightarrow+\infty}\frac{1}{n}\log\left(\lo(g^n\circ\gamma)\right)\leq\h(g).
\end{equation}
\end{theorem}

Looking carefully at the proof in Yomdin \cite{yomdin}, one sees that this result can be improved to the case when we consider a family of $C^r$-arcs $(\gamma_j)_j$ whose derivatives are uniformly bounded to the order $r$, that is, there is a positive number $K$ such that~$\|\gamma_j^{(k)}(t)\|\leq K$ for all $j$, $t\in [0,1]$ and $k\leq r$. Under these assumptions we have
\begin{equation}\label{yomdin-famille}
\limsup_{n\rightarrow+\infty}\frac{1}{n}\log\left(\max_j\left\{\lo(g^n\circ\gamma_j)\right\}\right)\leq\h(g)+\frac{2}{r}\liap(g).
\end{equation}

We also use the following lemma, which can be found in Gromov {\cite[3.3]{gromov-aft-yomdin}}.

\begin{lemma}[Gromov]\label{lemme-gromov}
Let $Y$ be the intersection of an algebraic affine curve in~$\R^d$ with $[-1,1]^d$. For any $r\in\N^*$ there exist at most $m_0$ $C^r$-arcs $\gamma_j:[0,1]\to Y$, where~$m_0$ is an integer depending only on $d$, $r$, and $\deg(Y)$, such that
\begin{itemize}
\item[(1)] $Y=\bigcup_j \gamma_j([0,1])$;
\item[(2)] $\|\gamma_j^{(k)}(t)\|\leq 1$ for all $j$, $t\in[0,1]$ and $k\leq r$;
\item[(3)] all $\gamma_j$'s are analytic diffeomorphisms from $(0,1)$ to their images;
\item[(4)] the images of the $\gamma_{j}$'s can only meet on their boundaries.
\end{itemize}
\end{lemma}

\begin{proof}[Proof of Theorem \ref{iteres-mvolr}]
Inequality (\ref{ent-r}) does not depend on the choice of a particular metric on $X$, so we can consider an embedding $X\subset\P^d_\R$ and take the metric induced by Fubini--Study on $X$. The projective space $\P^d(\R)$ is covered by the $(d+1)$ cubes~$Q_k$, $k\in\{0,\cdots,d\}$ given in homogeneous coordinates by $|Z_k|=\max_j|Z_j|$. Each of these $Q_k$ is located in the affine chart $U_k=\{Z_k\neq 0\}\simeq\R^d$, and in this chart it is identified with $[-1,1]^d$.

The degree of $D$ as a subvariety of $\P^d$ only depends on the Chern class $[D]$. Therefore we can apply Lemma \ref{lemme-gromov} to any divisor $D'\in\mathcal{V}(D)$, intersected with one  of the $Q_k$'s: any real locus of $D'\in\mathcal{V}(D)$ is covered by at most $m_1$ $C^r$-arcs $\gamma_{D',j}$, the integer $m_1=(d+1)m_0$ being independent of $D'$, such that $\|\gamma_{D',j}^{(k)}\|_\infty\leq K$ for all $k\leq r$, where $r$ is a fixed positive integer and $K$ a positive constant (which comes from the comparison of Euclidean and Fubini--Study metrics on $[-1,1]^d$). Now we apply (\ref{yomdin-famille}) to obtain
\begin{equation}
\begin{split}
\limsup_{n\to+\infty}\frac{1}{n}\log\left(\mvol_\R(f_*^nD)\right)
&\leq
\limsup_{n\to+\infty}\frac{1}{n}\log\left(m_1 \max_{D',j}\left\{\lo(f_\R^n\circ\gamma_{D',j})\right\}\right)\\
&\leq
\h(f_\R)+\frac{2}{r}\liap(f_\R).
\end{split}
\end{equation}
Since the regularity of both $X(\R)$ and $f_\R$ is $C^\infty$, we may take the limit as $r$ goes to $+\infty$ and get the desired inequality.
\end{proof}

\begin{rema}
Yomdin's theorem (as well as its version in family) and Gromov's lemma still hold for arbitrary dimensional submanifolds. Therefore the proof of Theorem \ref{iteres-mvolr} can be adapted when $X$ is a variety with higher dimension.
\end{rema}


\subsection{An upper bound for concordance}

\begin{theorem}\label{entropie}
Let $X$ be a real algebraic surface and let $f$ be a real hyperbolic type\footnote{Recall that \emph{hyperbolic type} just means that $\h(f_\C)>0$.} automorphism of $X$. Then 
\begin{equation}\label{dist-ent}
\alpha(X)\leq\frac{\h(f_\R)}{\h(f_\C)}.
\end{equation}
\end{theorem}

\begin{proof}
Let $\alpha$ be an exponent in the interval $\mathcal{A}(X)$. This means that there are~${q\in\N^*}$ and~$C>0$ such that $\mvol_\R(D)\geq C\vol_\C(D)^\alpha$ for all real ample divisors $D$ with~$[D]$~$q$-divisible. For such a divisor, $f_*^n[D]$ is also $q$-divisible for all $n\in\N$, and by Theorems \ref{iteres-mvolc} and \ref{iteres-mvolr} we get
\begin{equation}
\begin{split}
\h(f_\R)
&\geq
\limsup_{n\to+\infty}\frac{1}{n}\log\mvol_\R(f_*^nD)\\
&\geq
\limsup_{n\rightarrow+\infty}\frac{1}{n}\left(\log C+\alpha\log\vol_\C(f_*^nD)\right)\\
&=
\alpha\h(f_\C).
\end{split}
\end{equation}
Then we take the limit as $\alpha\to\alpha(X)$ and we obtain (\ref{dist-ent}).
\end{proof}


\subsection{A lower bound for real volume of the iterates of a divisor}

\begin{definition}\label{veryample}
Let $M$ be a differentiable surface. A family $\Gamma$ of curves on $M$ is said to be \emph{very ample} if for all $P\in M$ and for all directions $\mathcal{D}\subset {\rm T\!}_xM$, there is a curve $\gamma\in\Gamma$ on which $P$ is a regular point and whose tangent direction at $P$ is $\mathcal{D}$.
\end{definition}

\begin{example}\label{ex-tres-ample}
Let $X$ be a real algebraic surface and $D$ be a very ample real divisor on $X$. Then the family $\mathcal{V}(D)$, as a family of curves on $X(\R)$, is a very ample family in the sense of Definition \ref{veryample}.
\end{example}

\begin{theorem}\label{horseshoe}
Let $M$ be a compact Riemannian surface, $g:M\rightarrow M$ be a diffeomorphism of class $C^{1+\epsilon}$ (with $\epsilon>0$) with positive entropy, and $\Gamma$ be a very ample family of curves on $M$. Then for all $\lambda<\exp(\h(g))$, there exist a curve~${\gamma\in\Gamma}$ and a constant $C>0$ such that
\begin{equation}\label{long-exp}
\lo(g^n(\gamma))\geq C\lambda^n
\end{equation}
for all $n\in\N$.
\end{theorem}

In other words, we have the following inequality:
\begin{equation}\label{type-newhouse}
\sup_{\gamma\in\Gamma}\left\{\liminf_{n\rightarrow+\infty}\frac{1}{n}\log\left(\lo(g^n(\gamma))\right)\right\}\geq\h(g).
\end{equation}

This has to be compared with a similar result due to Newhouse \cite{newhouse}, who considers manifolds of arbitrary dimension and noninvertible maps, but obtains the inequality (\ref{type-newhouse}) with a limit superior instead of a limit inferior (assumptions on the family $\Gamma$ are also lightly different). On the other hand, the lower bound (\ref{type-newhouse}) is optimal when $M$ and $g$ are $C^\infty$, by Yomdin's Theorem \ref{yomdin}.

\begin{corollary}\label{minorer-mvolr}
Let $f$ be a real automorphism of a real algebraic surface $X$.
For all $\lambda<\exp(\h(f_\R))$ and all very ample real divisors $D$ on $X$, there exists $C>0$ such that
\begin{equation}
\mvol_\R(f_*^nD)\geq C\lambda^n
\end{equation}
for all $n\in\N$.
\end{corollary}

The proof of Theorem \ref{horseshoe} relies on a result due to Katok \cite[S.5.9 p. 698]{katok-hasselblatt}, which asserts that the entropy of surface diffeomorphisms is well approximated by horseshoes. For definition and properties of horseshoes, we refer to \cite[\textsection 6.5]{katok-hasselblatt}.

\begin{theorem}[Katok]\label{katok}
Let $M$ be a compact surface, and $g:M\rightarrow M$ be a diffeomorphism of class $C^{1+\epsilon}$ (with $\epsilon>0$) with positive entropy. For any $\eta>0$, there exists a horseshoe $\Lambda$ for some positive iterate $g^k$ of $g$ such that
\begin{equation}\label{h-approx}
\h(g)\leq\frac{1}{k}\h(g^k_{|\Lambda})+\eta.
\end{equation}
\end{theorem}

\begin{proof}[Proof of Theorem \ref{horseshoe}]
Fix real numbers $\lambda$ and $\eta$ such that $1<\lambda<\exp(\h(g))$ and $0<\eta\leq\h(g)-\log(\lambda)$. Let $\Lambda$ be a horseshoe for $G=g^k$ satisfying~(\ref{h-approx}). Let~$\Delta\supset\Lambda$ be a ``rectangle'' corresponding to this horseshoe, in such a way that ${\Lambda=\bigcap_{j\in\Z}G^j(\Delta)}$.
The set $G(\Delta)\cap\Delta$ has $q$ connected components $\Delta_1,\cdots,\Delta_q$, which are ``subrectangles'' crossing entirely $\Delta$ downward (see Figure \ref{ex-horseshoe}). The restriction~$G_{|\Lambda}$ is topologically conjugate to the full-shift on $q$ symbols, by the conjugacy map
\begin{equation*}
\begin{split}
\{1,\cdots,q\}^\Z
&\longrightarrow
\Lambda\\
(\omega_j)_{j\in\Z}
&\longmapsto
\bigcap_{j\in\Z}G^j(\Delta_{\omega_j}).
\end{split}
\end{equation*}
In particular $\h(G_{|\Lambda})=\log(q)$. We denote by $L$ the distance between the upper and lower side of $\Delta$.

\begin{figure}[h]
\begin{center}
\scalebox{0.5}{\input{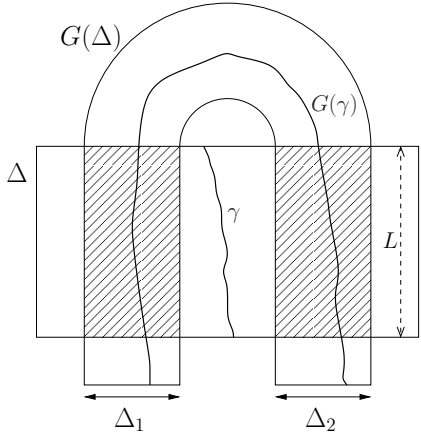}}
\caption{An example of horseshoe, here with $q=2$.}
\label{ex-horseshoe}
\end{center}
\end{figure}

\begin{lemma}\label{q^n}
Let $\gamma\subset\Delta$ be an arc crossing the rectangle $\Delta$ downward. Then $\lo(G^n(\gamma))\geq q^nL$ for all $n\in\N$.
\end{lemma}

\begin{proof}
It is enough to remark that the arc $G^n(\gamma)$ contains $q^n$ subarcs crossing $\Delta$ downward (see Figure \ref{ex-horseshoe} for $n=1$). This can be seen by induction on $n$.
\end{proof}

Now fix a point $P\in\Lambda$, $P=\bigcap_{j\in\Z}G^j(\Delta_{\omega_j})$. Let $\gamma\in\Gamma$ be a curve that goes through $P$ transversally to the stable variety $W^s(P)$ (the horizontal one). For any sequence $(\epsilon_j)_{j\in\N}\in\{1,\cdots,q\}^\N$, we set (see Figure \ref{ex-horseshoe2})
\begin{equation}
R_{\epsilon_1,\cdots,\epsilon_n}=\bigcap_{j=0}^{n}G^{-j}(\Delta_{\epsilon_{j}}).
\end{equation}

\begin{figure}[h]
\begin{center}
\scalebox{0.7}{\input{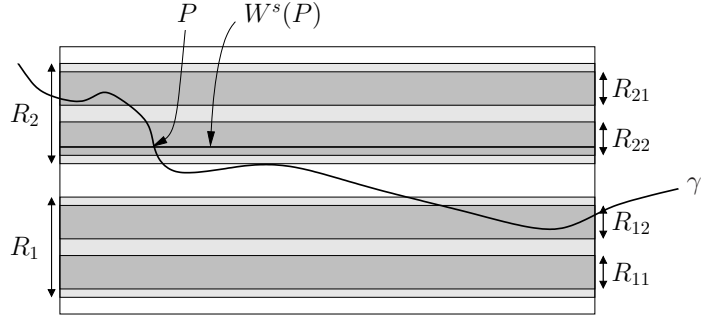}}
\caption{The rectangles $R_{\epsilon_1}$ and $R_{\epsilon_1,\epsilon_2}$ for the horseshoe of Figure \ref{ex-horseshoe}.}
\label{ex-horseshoe2}
\end{center}
\end{figure}

The sequence $(R_{\epsilon_1,\cdots,\epsilon_n})_{n\in\N}$ is a decreasing sequence of nested rectangles that converge to the curve $\bigcap_{j\in\N}G^{-j}(\Delta_{\epsilon_{j}})$. If $(\epsilon_n)_{n\in\N}=(\omega_{-n})_{n\in\N}$ this curve is the stable variety $W^s(P)$ (intersected with $\Delta$). Since $\gamma$ is transverse to it, there exist an integer $n_0$ and a subarc $\gamma'\subset\gamma$ such that $\gamma'$ crosses the rectangle $R_{\omega_0,\cdots,\omega_{-n_0}}$ downward. (On Figure \ref{ex-horseshoe2}, we may choose $\gamma'\subset R_{22}$.) Hence the arc ${G^{n_0}(\gamma')\subset G^{n_0}(\gamma)}$ satisfies the assumptions of Lemma \ref{q^n}, and thus $\lo(G^{n_0+n}(\gamma))\geq q^nL$ for all~${n\in\N}$. So if we set $C'=\min\left\{\frac{L}{q^{n_0}},\left(\frac{\lo(G^n(\gamma))}{q^n}\right)_{0\leq n\leq n_0-1}\right\}$, then
\begin{equation}
\begin{split}
\lo(g^{nk}(\gamma))
&\geq C'q^n\\
&= C'\exp(n\h(g^k_{|\Lambda}))\\
&\geq C'\exp(nk(\h(g)-\eta))\\
&\geq C'\lambda^{nk}.
\end{split}
\end{equation}
Since $\lo(g^n(\gamma))\leq\|{\rm D}g^{-1}\|\lo(g^{n+1}(\gamma))$, we get Inequality (\ref{long-exp}) by Euclidean division by $k$, where we have set $C=C'(\lambda\|{\rm D}g^{-1}\|_\infty)^{-k}>0$.
\end{proof}


\subsection{An exact formula for concordance when $\rho(X_\R)=2$}

\begin{theorem}\label{formule-exacte}
Let $X$ be a real algebraic surface with $\rho(X_\R)=2$. Assume that there exists a real hyperbolic type automorphism $f$ on $X$. Then
\begin{equation}\label{exacte}
\alpha(X)=\frac{\h(f_\R)}{\h(f_\C)}.
\end{equation}
\end{theorem}

\begin{rema}
The assumptions of the theorem imply that the surface $X$ is either a torus, a K3 surface, or an Enriques surface. Indeed, as seen in Section \ref{intro-dyn}, its minimal model is either one of these three types of surfaces, or a rational surface. But if $X$ is not minimal or if $X$ is rational, then the class of the canonical divisor~$K_X$ would be nontrivial in $\NS(X_\R;\R)$. Since this class is preserved by $f_*$, this map would have $1$ as an eigenvalue. This is impossible, because $\NS(X_\R;\R)$ has dimension $2$ and the spectral radius of $f_*$ must be $>1$.
\end{rema}

\begin{proof}[Proof of Theorem \ref{formule-exacte}]
By Theorem \ref{entropie}, it is enough to prove that any nonnegative exponent $\alpha<\frac{\h(f_\R)}{\h(f_\C)}$ belongs to $\mathcal{A}(X)$. This is obvious when $\h(f_\R)=0$, so we suppose that $f_\R$ has positive entropy, and we fix such an exponent $\alpha$.

\begin{lemma}\label{une orbite}
Let $D$ be a very ample real divisor on $X$. There exists $C > 0$ such that
\begin{equation}\label{eq-une-orbite}
\mvol_\R(f_*^nD) \geq C \vol_\C(f_*^nD)^\alpha
\end{equation}
for all $n \in \Z$.
\end{lemma}

\begin{proof}
Since $\lambda(f)^\alpha=\exp(\alpha\h(f_\C))<\exp(\h(f_\R))$, there exists, by Corollary~\ref{minorer-mvolr}, a positive number $C_\R$ such that $\mvol_\R(f_*^nD)\geq C_\R\lambda(f)^{n\alpha}$ for all $n\in\N$. On the other hand there is a positive number $C_\C$ such that $\vol_\C(f_*^nD)\leq C_\C\lambda(f)^n$ for all ${n\in\N}$ (cf (\ref{volc-borne})). It follows that $\mvol_\R(f_*^nD) \geq C^+ \vol_\C(f_*^nD)^\alpha$ for all $n\in\N$, where we have set $C^+=C_\R/C_\C^\alpha$.

Applying the same argument to $f^{-1}$, there exists a positive number $C^-$ such that ${\mvol_\R(f_*^{-n}D) \geq C^- \vol_\C(f_*^{-n}D)^\alpha}$ for all $n\in\N$. Hence we obtain (\ref{eq-une-orbite}), with~${C=\min(C^+,C^-)}$.
\end{proof}

\begin{lemma}\label{nb fini d'orbites}
There are finitely many real ample divisors $D_1,\cdots,D_r$ on $X$ such that any real ample divisor $D$ on $X$ is equivalent to one of the form $\sum_{k=1}^{s} f_*^nD_{j_k}$, with $n\in\Z$ and $j_k\in\{1,\cdots,r\}$.
\end{lemma}

\begin{proof}
On the $2$-dimensional space $\NS(X_\R;\R)$, the isometry $f_*$ has exactly two eigenlines $\mathcal{D}^+$ and $\mathcal{D}^-$, respectively associated with eigenvalues $\lambda(f)$ and $\lambda(f)^{-1}$. These lines are necessarily the isotropic directions of the intersection form. We choose eigenvectors $\theta^+\in\mathcal{D}^+$ and $\theta^-\in\mathcal{D}^-$ in the closure of $\Pos(X_\R)$, so that this cone is bordered by half-lines $\R^+ \theta^+$ and $\R^+ \theta^-$. Since it is preserved by $f_*$, the cone $\Amp(X_\R)$ coincides with $\Pos(X_\R)$. The integer points in this cone correspond to classes of real ample divisors. Let $\theta_1$ be such a point that we choose to be primitive, and let $\theta_2 = f_* \theta_1$ (observe that $\theta_2$ is also primitive). Denote by $\mathcal{D}$ the closed convex cone of $\NS(X_\R;\R)$ bordered by half-lines $\R^+\theta_1$ and $\R^+\theta_2$. By construction $\mathcal{D} \backslash \{0\}$ is a fundamental domain for the action of $f_*$ on $\Amp(X_\R)$ (see Figure \ref{dom-fond}).

\begin{figure}[h]
\begin{center}
\scalebox{0.4}{\input{fig3.pstex_t}}
\caption{The fundamental domain $\mathcal{D}\backslash\{0\}$.}\label{dom-fond}
\end{center}
\end{figure}

Denote by $\theta_3,\theta_4,\cdots,\theta_r$ the entire points inside the parallelogram whose vertices are $0$, $\theta_1$, $\theta_1+\theta_2$ and $\theta_2$. Any point in $\mathcal{D}$ can be expressed uniquely as $k_1\theta_1 + k_2\theta_2 + \theta_j$ or $k_1\theta_1+k_2\theta_2$, with $(k_1,k_2)\in\N^2$ and $j \in \{3,\cdots,r\}$. For all real ample divisors~$D$, there is $n\in\Z$ such that $f_*^{-n}[D] \in \mathcal{D}$, so we are done by setting $D_1,\cdots,D_r$ real ample divisors whose classes are $\theta_1,\cdots,\theta_r$.
\end{proof}

We go back to the proof of Theorem \ref{formule-exacte}. Let $q$ be a positive integer such that the divisors $D'_1=qD_1,\cdots,D'_r=qD_r$ are all very ample. By Lemma \ref{une orbite}, there exists a positive number $C$ such that, for all $j \in \{1,\cdots,r\}$ and $n \in \Z$, we have $\mvol_\R(f_*^nD'_j) \geq C \vol_\C(f_*^nD'_j)^\alpha$. Let $D$ be a real ample divisor whose Chern class is $q$-divisible. There are $n \in \Z$ and $j_1,\cdots,j_s \in \{1,\cdots,r\}$ such that~${[D] = \sum_{k=1}^sf_*^n[D'_{j_k}]}$. Then
\begin{align}
\mvol_\R(D) & \geq \sum_k\mvol_\R(f_*^nD'_{j_k})\\
& \geq C \sum_k\vol_\C(f_*^nD'_{j_k})^\alpha\label{ligne1}\\
& \geq C \left(\sum_k\vol_\C(f_*^nD'_{j_k})\right)^\alpha\label{ligne2}\\
& = C \vol_\C(D)^\alpha.
\end{align}
From (\ref{ligne1}) to (\ref{ligne2}), we have used the following special case of Minkowski inequality:
\begin{equation}\label{minkowski}
\left(\sum_{k=1}^s|x_k|\right)^\alpha\leq\sum_{k=1}^s|x_k|^\alpha\quad\quad\forall\alpha\in(0,1].
\end{equation}

Hence we see that $\alpha$ belongs to $\mathcal{A}(X)$, and Theorem \ref{formule-exacte} is proved.
\end{proof}

\begin{rema}
We do not know if concordance is achieved in Theorem \ref{formule-exacte}.
\end{rema}



\section{Abelian Surfaces}\label{tores}


\subsection{Preliminaries}

A {\it real abelian variety} $X$ is a real algebraic variety whose underlying complex manifold $X(\C)$ is a complex torus $\C^g/\Lambda$. We say \emph{real elliptic curve} when $g=1$, and \emph{real abelian surface} when $g=2$. As we still assume that~$X(\R)\neq\emptyset$, we are brought to the case where the antiholomorphic involution~$\sigma_X$ comes from the complex conjugation on $\C^g$, and the lattice $\Lambda$ has the form
\begin{equation}\label{reseau}
\Lambda=\Z^g\oplus\tau\Z^g,
\end{equation}
where $\tau\in\M_g(\C)$ is such that $\Im(\tau)\in\GL_g(\R)$ and $2\Re(\tau)=\begin{pmatrix}I_r&0\\0&0\end{pmatrix}$, the integer~$r$ being characterized by the fact that $X(\R)$ has $2^{g-r}$ connected components (cf \cite[\textsection IV]{silhol}).

A (real) {\it homomorphism} between two real abelian varieties is a holomorphic map~$f:X=\C^g/\Lambda\rightarrow X'=\C^{g'}/\Lambda'$ which is compatible with the real structures (that is, $\sigma_{X'}\circ f=f\circ\sigma_X$) and which respects the abelian group structures (this is equivalent to $f(0)=0$). Such a map lifts to a unique $\C$-linear map~${F:\C^g\to\C^{g'}}$ such that $F(\Lambda)\subset\Lambda'$, whose matrix has integer coefficients (for ${F(\Z^g)\subset\Lambda'\cap\R^{g'}=\Z^{g'}}$). We also talk about \emph{endomorphisms}, \emph{isomorphisms} and \emph{automorphisms} of real abelian varieties. Observe that in this context,  automorphisms are asked to \emph{preserve the origin}.

A (real) {\it isogeny} between two real abelian varieties of same dimension is a homomorphism of real abelian varieties that is surjective, which means that its matrix has maximal rank. Two real abelian varieties are said to be {\it isogenous} when there exists an isogeny from one to the other (this is an equivalence relation; cf \cite[1.2.6]{birkenhake-lange}).

\begin{rema}
The real Picard number does not change by isogeny. Indeed, any isogeny $f:X\rightarrow X'$ gives rise to a homomorphism ${f^*:\NS(X'_\R;\Z)\rightarrow\NS(X_\R;\Z)}$ which is injective, hence $\rho(X'_\R)\leq\rho(X_\R)$; the other inequality follows by the symmetry of the isogeny relation.
\end{rema}

\begin{lemma}\label{isog}
Any real abelian surface $X$ is isogenous to $\C^2/\Lambda$, where $\Lambda$ has the form
\begin{equation}\label{reseau2}
\Lambda=\Z^2\oplus iS\Z^2,
\end{equation}
the matrix $S = \begin{pmatrix} y_1 & y_3 \\ y_3 & y_2 \end{pmatrix}$ being symmetric positive definite. We then have
\begin{equation}\label{nbpicard-ab}
\rho(X_\R) = 4 - \dim_\Q ( \Q y_1 + \Q y_2 + \Q y_3 ).
\end{equation}
\end{lemma}

\begin{proof}
The existence of a real polarization on $X=\C^g/\Lambda$ (see \cite[\textsection IV.3]{silhol}) implies that the lattice $\Lambda$ can be set on the form $D\Z^2\oplus\tau\Z^2$, the matrix $D$ being diagonal with integer coefficients, and the matrix $\tau$ being symmetric, with $S=\Im(\tau)$ positive definite and $2\Re(\tau)$ an integer matrix. Hence the dilation by $2$ in $\C^2$ gives rise to a real isogeny from $\C^2 / \Lambda$ to $\C^2 / (\Z^2 \oplus iS\Z^2)$. Equality (\ref{nbpicard-ab}) comes from \cite[\textsection 1, 3.4]{birkenhake-lange} and  \cite[\textsection IV (3.4)]{silhol}.
\end{proof}

\begin{rema}
As a consequence of (\ref{nbpicard-ab}), we see that the real Picard number $\rho(X_\R)$ is $1$, $2$ or $3$. In contrast, the complex Picard number $\rho(X_\C)$ can also achieve the extra value $4$, when $X$ is isogenous to the square of an elliptic curve with complex multiplication (cf \cite[\textsection 2 7.1]{birkenhake-lange}).
\end{rema}

Now observe the following fact, which is very specific to tori.

\begin{proposition}\label{alpha-leq-1/2}
Let $f$ be an automorphism of a real abelian surface $X$. Then
\begin{equation}
\h(f_\C)=2\h(f_\R).
\end{equation}
Accordingly, $\alpha(X)\leq 1/2$ as soon as $X$ admits real hyperbolic type automorphisms.
\end{proposition}

\begin{proof}
We lift the automorphism $f$ to a $\C$-linear map $F:\C^2\to\C^2$ whose matrix is in $\SL_2(\Z)$ (replacing $f$ by $f^2$ if necessary). If $F$ has spectral radius~$1$, then it is obvious that $\h(f_\R)=\h(f_\C)=0$. Otherwise, $F$ has two distinct eigenvalues $\lambda$ and~$\lambda^{-1}$, with $|\lambda|>1$. As a $\R$-linear map of $\C^2$, $F$ has eigenvalues~${(\lambda,\lambda,\lambda^{-1},\lambda^{-1})}$ (with multiplicities), thus $\h(f_\C)=2\log|\lambda|$ (see, for instance, \cite[2.6.4]{brin-stuck}). Restricted to $\R^2$, $F$ has eigenvalues $(\lambda,\lambda^{-1})$, hence $\h(f_\R)=\log(|\lambda|)$.

The last part is a consequence of Theorem \ref{entropie}.
\end{proof}

The aim of what follows is to prove the following theorem, which describes exhaustively the concordance for real abelian surfaces.

\begin{theorem}\label{thm-surf-ab}
Let $X$ be a real abelian surface. We have the following alternative:
\begin{itemize}
\item[(1)] $\rho(X_\R)=1$ and $\alpha(X)=1$;
\item[(2)] $\rho(X_\R)=2$ and
\begin{itemize}
\item[(i)] if the intersection form represents $0$ on $\NS(X_\R;\Z)$, then $\alpha(X)=1$,
\item[(ii)] otherwise, $\alpha(X)=1/2$;
\end{itemize}
\item[(3)] $\rho(X_\R)=3$ and $\alpha(X)=1/2$.
\end{itemize}
The concordance is achieved in all cases. It equals $1/2$ if and only if $X$ admits real hyperbolic type automorphisms.
\end{theorem}

We already dealt with the case $\rho(X_\R)=1$ (cf Corollary \ref{rho=1}), so we  focus on the last two cases.


\subsection{Invariance of concordance under isogeny}

\begin{proposition}\label{inv-iso}
Let $X$ and $X'$ be two isogenous real abelian varieties. Then ${\mathcal{A}(X)=\mathcal{A}(X')}$, and consequently $\alpha(X)=\alpha(X')$.
\end{proposition}

\begin{proof}
Since the isogeny relation is symmetric, it is enough to show the inclusion~$\mathcal{A}(X)\subset\mathcal{A}(X')$. Let $f: X' \rightarrow X$ be an isogeny. For $\K=\R$ or $\C$, denote by~$f_\K$ the induced map from $X'(\K)$ to $X(\K)$. We take an arbitrary K\"ahler metric on $X$, and then we take its pullback on $X'$, so that $f$ is locally an isometry for the respective metrics.

Fix any $\alpha\in\mathcal{A}(X)$. There exist $C>0$ and $q\in\N^*$ such that any real ample divisor $D$ on $X$, whose Chern class is $q$-divisible, satisfies $\mvol_\R(D)\geq C\vol_\C(D)^\alpha$. Since $f^*:\NS(X_\R;\Z)\rightarrow\NS(X'_\R;\Z)$ is an injective homomorphism, its image has finite index $n$. Let $D'$ be a real ample divisor on $X'$ whose class is $nq$-divisible. Then there is a real ample divisor $D$ on $X$ with $[D']=[f^*D]$, and furthermore $[D]$ is $q$-divisible.

Any point on the curve $D(\R)$ has exactly $\deg(f_\R)$ preimages, hence  $\vol_\R(f^*D)=\deg(f_\R)\vol_\R(D)$ by the choice of the metrics. Since $f^*$ realizes a bijective map between $\mathcal{V}(D)$ and $\mathcal{V}(D')$, we deduce, taking the upper bound on $\mathcal{V}(D)$, that ${\mvol_\R(D')=\deg(f_\R)\mvol_\R(D)}$.

By the same argument, we also have $\vol_\C(D')=\deg(f_\C)\vol_\C(D)$. So if we set~$C'=C\deg(f_\R)/{\deg(f_\C)}^\alpha$, we obtain $\mvol_\R(D') \geq C'\vol_\C(D')^\alpha$. This shows that the exponent $\alpha$ is contained in $\mathcal{A}(X')$.
\end{proof}


\subsection{Picard number 2}


\subsubsection{Hyperbolic rank 2 lattices}\label{section-reseaux}

By definition, a \emph{lattice} is a free abelian group $L$ of finite rank, equipped with a nondegenerate symmetric bilinear form $\varphi$ taking integral values. We say the lattice is \emph{hyperbolic} when the signature of the induced quadratic form on $L\otimes\R$ is $(1,\rg(L)-1)$. The determinant of the matrix of $\varphi$ in a base of $L$ is the same for all bases. Its absolute value is a positive integer, called the \emph{discriminant} of the lattice.

Let $(L,\varphi)$ be a rank 2 hyperbolic lattice. Then $L\otimes\R$ has exactly two isotropic lines. The discriminant $\delta$ is a perfect square if and only if the quadratic form associated to $\varphi$ represents $0$, which means that there exists a nonzero isotropic point in $L$, or to say it otherwise both isotropic lines in $L\otimes\R$ are rational.

Suppose that $\delta$ is no perfect square. The study of Pell--Fermat equation then implies the existence of a \emph{hyperbolic isometry} of $L$, that is, an isometry whose spectral radius is greater than $1$. Such an isometry spans a finite index subgroup of the isometries of $L$. To be more precise, the group $\SO(L,\varphi)$ of direct isometries of $L$ (those with determinant $1$) is an abelian group isomorphic to $\Z\times\Z/2\Z$, and any infinite order element in $\SO(L,\varphi)$ is hyperbolic.

Conversely if $\delta$ is a perfect square, there is no hyperbolic isometry, and the isometry group is finite. More precisely,  $\SO(L,\varphi)=\{\id,-\id\}\simeq\Z/2\Z$.

\begin{example}
Let $X$ be a real algebraic surface with $\rho(X_\R)=2$. Then the group~$\NS(X_\R;\Z)$, equipped with the intersection form, is a rank 2 hyperbolic lattice.
\end{example}


\subsubsection{Surfaces with real elliptic fibrations}

Let $X$ and $Y$ be two complex algebraic varieties. An \emph{elliptic fibration} on $X$ is a holomorphic map $\pi:X\to Y$ that is proper and surjective, and such that the generic fiber is an elliptic curve. When the varieties $X$, $Y$ and the morphism $\pi$ are defined over $\R$, we call the elliptic fibration \emph{real}.

\begin{proposition}
Let $X$ be a real abelian surface with $\rho(X_\R)=2$. The following are equivalent:
\begin{itemize}
\item[(1)]
the intersection form on $\NS(X_\R;\Z)$ represents $0$;
\item[(2)]
there exists a real elliptic fibration on $X$;
\item[(3)]
$X$ is isogenous to the product of two elliptic curves $E_1$ and $E_2$.
\end{itemize}
In this case, the concordance of $X$ equals $1$, and it is achieved.
\end{proposition}

\begin{rema}
The elliptic curves $E_1$ and $E_2$ cannot be isogenous, for otherwise the Picard number would be $3$.
\end{rema}

\begin{proof}
$(1)\Rightarrow (2)$: Let $\theta$ a nonzero primitive point in $\NS(X_\R;\Z)$ such that ${\theta^2=0}$. After changing $\theta$ into $-\theta$ if necessary, there exists a real effective and irreducible divisor $D$ whose class is $\theta$ (here we use the fact that $\Nef(X_\R)$ is the closure of~$\Pos(X_\R)$; cf \cite[1.5.17]{lazarsfeld}). By the genus formula, the arithmetic genus of~$D$ is $1$. Since an abelian surface does not have any rational curve, $D$ must be a real elliptic curve. We may suppose that $D$ goes through $0$ (if not, we translate it and obtain an equivalent divisor), and thus it is a real subtorus. Now the canonical projection ${\pi:X\to X/D}$ is a real elliptic fibration.

$(2)\Rightarrow (3)$ follows from the Poincar\'e reducibility theorem (see \cite[\textsection VI 8.1]{debarre}).

$(3)\Rightarrow (1)$: Let $f:X\to E_1\times E_2$ be an isogeny. The effective divisor $D$ given by~$f^*(E_1\times\{0\})$ has self-intersection $0$, so the intersection form represents $0$.

As the nef cone of $E_1\times E_2$ is spanned by $[E_1\times\{0\}]$ and $[\{0\}\times E_2]$, the interval $\mathcal{A}(E_1\times E_2)$ is equal to $[0,1]$, by Proposition \ref{conenef}. By invariance under isogeny, we also have $\mathcal{A}(X)=[0,1]$.
\end{proof}


\subsubsection{Surfaces with no real elliptic fibration}

\begin{theorem}\label{delta}
Let $X$ be a real abelian surface with $\rho(X_\R)=2$. Assume that the intersection form on $\NS(X_\R;\Z)$ does not represent $0$. Then
\begin{itemize}
\item[(1)]
there exists a real hyperbolic type automorphism on $X$;
\item[(2)]
the concordance of $X$ equals $1/2$ and it is achieved.
\end{itemize}
\end{theorem}

We use the following result (see, for instance, \cite{X}):

\begin{theorem}[Torelli theorem for tori]\label{torelli-tori}
Let $X$ be a real abelian surface and let $\phi$ be a Hodge isometry of ${\rm H}^2(X(\C);\Z)$ that preserves the ample cone and has determinant $+1$. Then there exists a \emph{complex} automorphism $f$ of $X(\C)$, unique up to a sign, such that $f_*=\phi$. If moreover $\phi$ commutes with the involution $\sigma_X^*$, then~$f$ or $f^2$ is a \emph{real} automorphism.
\end{theorem}

\begin{rema}
To prove the last part, it is enough to remark that $\sigma_X\circ f\circ\sigma_X=\pm f^{-1}$ by the uniqueness part.
\end{rema}

\begin{lemma}\label{lemme-alg-com}
Let $L$ be an free abelian group of finite rank, $L'$ be a finite index subgroup of $L$ and $\phi'$ be an automorphism of $L'$. Then some positive iterate $\phi'^k$ extends to an automorphism $\phi$ on $L$.
\end{lemma}

\begin{proof}
Denote by $q$ the exponent of the group $L/L'$, so that  $qL\subset L'$. As $\phi'$ projects to an automorphism of $L'/qL'$ that has finite order $k$, it follows that~${\phi'^k(qL)\subset qL}$. Let $\mu_q:L\to qL$ be the isomorphism defined by $\theta\mapsto q\theta$. Then the automorphism~${\phi=\mu_q^{-1}\circ{\phi'^k}_{|qL}\circ\mu_q}$ satisfies the desired property.
\end{proof}

\begin{proof}[Proof of Theorem \ref{delta}]
Since the intersection form does not represent $0$, there exists a hyperbolic isometry $\phi_1$ of $L_1=\NS(X_\R;\Z)$ (cf \textsection \ref{section-reseaux}). Replacing $\phi_1$ by~$\phi_1^2$ if necessary, we may suppose that $\phi_1$ preserves the cone $\Amp(X_\R)$ and that $\det(\phi_1)=1$. Denote by $L_2$ the orthogonal of $L_1$ in $L={\rm H}^2(X(\C);\Z)$, and by~$L'$ the direct sum $L_1\oplus L_2$. The subgroup $L'$ has finite index in $L$, and so, by Lemma~\ref{lemme-alg-com},~${\phi_1^k\oplus \id_{L_2}}$ extends to an automorphism $\phi$ on ${\rm H}^2(X(\C);\Z)$ for some~$k\in\N^*$. It is clear by construction that $\phi$ satisfies the assumptions of Theorem \ref{torelli-tori}. Thus there exists a real automorphism $f$ on $X$ such that $f_*=\phi^2$. Its entropy is positive, as a multiple of the spectral logradius of $\phi_1$.

The equality $\alpha(X)=1/2$ follows from Theorem \ref{formule-exacte} and Proposition \ref{alpha-leq-1/2}. In order to show that the concordance is achieved, we replace $\alpha$ by $\frac{1}{2}=\frac{\h(f_\R)}{\h(f_\C)}$ inside the proof of Theorem \ref{formule-exacte} by using the following lemma, which improves the inequality of Corollary \ref{minorer-mvolr}.
\end{proof}

\begin{lemma}
Let $f$ be a real automorphism of a real abelian surface $X$. Assume that $\lambda=\exp(\h(f_\R))>1$. Then for all very ample real divisors $D$ on $X$, there exists $C>0$ such that 
\begin{equation}
\mvol_\R(f_*^nD) \geq C \lambda^n
\end{equation}
for all $n\in\N$.
\end{lemma}

\begin{proof}
The automorphism $f$ lifts to a linear self-map $F$ of $\R^2$. Replacing $f$ with~$f^2$ if necessary, $F$ has eigenvalues $\lambda$ and $\lambda^{-1}$. We choose a scalar product on $\R^2$ such that the eigenlines $\mathcal{D}^+$ and $\mathcal{D}^-$, respectively, associated to $\lambda$ and $\lambda^{-1}$, are orthogonal. Then we take on $X(\R)$ the Riemannian metric induced by this scalar product.

If necessary, we change $D$ (by translation) into an equivalent divisor containing the origin as a smooth point. Then the curve $D(\R)$ contains a smooth simple arc $\gamma$ through $0$. Let $\tilde\gamma$ be the lift of $\gamma$ to $\R^2$ containing the origin, and let $p:\R^2\to\R^2$ be the projection on $\mathcal{D}^+$ with direction $\mathcal{D}^-$. Then 
\begin{equation}
\begin{split}
\mvol(f_*^nD)
&\geq
\lo(f^n(\gamma))\\
&=
\lo(F^n(\tilde\gamma))\\
&\geq
\lo(p\circ F^n(\tilde\gamma))\\
&=
\lambda^n\lo(p(\tilde\gamma)).
\end{split}
\end{equation}
Observe that $\lo(p(\tilde\gamma))>0$. Indeed, if it were zero, $\tilde\gamma$ would be contained in~$\mathcal{D}^-$, hence the curve $D(\R)$, being analytic, would contain the projection of $\mathcal{D}^-$ on the torus $\R^2/\Z^2$. But this is not the case, because the last one is Zariski-dense, the line $\mathcal{D}^-$ being irrational. Thus we get the result with $C=\lo(p(\tilde\gamma))$.
\end{proof}


\subsection{Picard number 3}

\begin{lemma}
Let $X$ be a real abelian surface with $\rho(X_\R)=3$. There exists a real elliptic curve $E$ such that $X$ is isogenous to $E\times E$.
\end{lemma}

\begin{proof}
Changing $X$ by isogeny if necessary, $X$ has the form $\C^2 / \Lambda$, where $\Lambda$ is like in Lemma \ref{isog}. As $\rho(X_\R)=3$, it follows that $\Q y_1+\Q y_2+\Q y_3$ has dimension $1$, so there exists~$m\in\N^*$ such that $my_2$ and $my_3$ are in $\Z y_1$ ($y_1\neq 0$, for~${y_1y_2-y_3^2=\det(S)>0}$). Then the dilation by $m$ in $\C^2$ gives rise to an isogeny from $X$ to $E\times E$, where $E$ is the elliptic curve $\C / (\Z \oplus iy_1 \Z)$.
\end{proof}

So we see that it is enough to restrict ourselves to the case $X=E\times E=\C^2/\Lambda$, where~$E=\C/(\Z\oplus\tau\Z)$ is a real elliptic curve, with $y=\Im(\tau)>0$ and $2\Re(\tau)\in\Z$, and $\Lambda$ is the lattice $\Z^2\oplus\tau\Z^2$. Furthermore, since the concordance is invariant under isogeny, we can even suppose that $\Re(\tau)=1/2$, so that the curve $E(\R)$ has only one connected component, which is identified with $\R/\Z$.

Observe that the group $\GL_2(\Z)$ acts on $X$ and gives many examples of real hyperbolic type automorphisms. A consequence from this fact and Proposition \ref{alpha-leq-1/2} is that we already have the inequality
\begin{equation}\label{leq1/2}
\alpha(X)\leq 1/2.
\end{equation}

In order to compute volumes, we choose the standard Euclidean metric on the torus $X=\C^2/\Lambda$, whose K\"ahler form is given by $\kappa= \frac{i}{2} \left( dz_1 \wedge d\b{z_1} + dz_2 \wedge d\b{z_2} \right)$. For this metric, we have $\vol_\R(E)=1$ and $\vol_\C(E)=y$. In the remaining part of this section, we follow \cite{christol}.

\begin{definition}
A \emph{rational line} on $X$ is the projection of a line of $\C^2$ given by an equation $az_1=bz_2$ with $(a,b)\in\Z^2$. The number $a/b\in\Q\cup\{\infty\}$ is the \emph{slope} of this rational line.
\end{definition}

\begin{example}
The curves $H$, $V$, and $\Delta$, which, respectively are the horizontal, the vertical, and the diagonal of $E\times E$, are rational lines with slopes by $0$, $\infty$, and~$1$, respectively. Their classes form a base of $\NS(X_\R;\Z)$ (one can use \cite[\textsection 1, 3.4]{birkenhake-lange} and~\cite[\textsection IV (3.4)]{silhol} to make this computation).
\end{example}

\begin{lemma}\label{vol-rat}
Let $D$ be a rational line on $X$. Then 
\begin{equation}
\vol_\R(D)=C\vol_\C(D)^{1/2},
\end{equation}
with $C=y^{-1/2}=\vol_\R(E)/\vol_\C(E)^{1/2}$.
\end{lemma}

\begin{proof}
Let $a/b$ be the slope of $D$, with coprime integers $a$ and $b$. We compute the length of $D(\R)$ by the Pythagorean theorem:
\begin{equation}
\vol_\R(D)=\sqrt{a^2+b^2}.
\end{equation}
On the other hand, it is clear, from the form of $\kappa$, that
\begin{equation}
\vol_\C(D)=\vol_\C(E)(D\cdot H+D\cdot V).
\end{equation}
We easily check that $D\cdot H=a^2$ and $D\cdot V=b^2$, hence
\begin{equation}
\vol_\C(D)=y(a^2+b^2).
\end{equation}
\end{proof}

The group $\SL_2(\Z)$ acts by automorphisms on $X$, thus by isometries on $\NS(X_\R;\R)$. This action preserves the ample cone $\Amp(X_\R)$, which here is the same as $\Pos(X_\R)$ (see \cite[\textsection 1.5.B]{lazarsfeld}).

If we identify the disk $\mathbf{D}=\P(\Amp(X_\R))$ with the Poincar\'e half-plane $\mathbf{H}$, by the unique isometry matching the class of a rational line in $\partial\mathbf{D}$ with the inverse of its slope in $\partial\mathbf{H}=\R\cup\{\infty\}$, then the induced action of $\PSL_2(\Z)$ on $\mathbf{D}$ corresponds to the standard action by homographies on $\mathbf{H}$. As a consequence we see that the triangle ${\mathbf{T} \subset \mathbf{D}}$, whose vertices are $\P[H]$, $\P[V]$, and $\P[\Delta]$, contains some fundamental domain for the action of $\PSL_2(\Z)$ on $\mathbf{D}$ (see Figure \ref{triangle}).\\

\begin{figure}[h]
\begin{center}
\scalebox{0.6}{\input{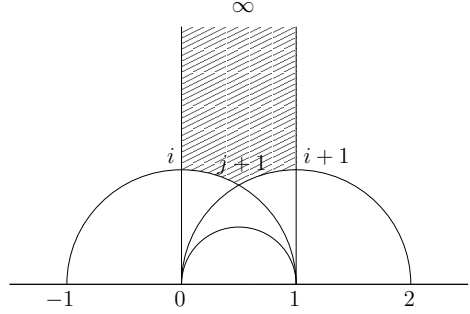}}
\caption{A fundamental domain for the action of $\PSL_2(\Z)$ on~$\mathbf{H}$. This domain is contained in the triangle corresponding to $\mathbf{T}$, whose vertices are $\infty$, $0$ and $1$.}\label{triangle}
\end{center}
\end{figure}

Let $D$ be a real ample divisor on $X$. There exists some $f\in\SL_2(\Z)$ such that~$\P(f_*^{-1} [D])\in\mathbf{T}$. So there are nonnegative numbers $k_1$, $k_2$ and $k_3$ such that~$f_*^{-1} [D] = k_1 [H] + k_2 [V] + k_3 [\Delta]$. The numbers $k_1$, $k_2$, and $k_3$ are actually integers, for $([H],[V],[\Delta])$ is a base of $\NS(X_\R;\Z)$. Hence the divisor $D$ is equivalent to $k_1D_1+k_2D_2+k_3D_3$, where $D_1=f(H)$, $D_2=f(V)$, and $D_3=f(\Delta)$ are rational lines. Then
\begin{align}
\mvol_\R(D) & \geq \sum_j k_j\vol_\R(D_j)\\
& = C \sum_j k_j\vol_\C(D_j)^{1/2}
&&\text{(by Lemma \ref{vol-rat})}\\
& \geq C \left(\sum_j {k_j}^2\vol_\C(D_j)\right)^{1/2}
&&\text{(by Minkowski inequality (\ref{minkowski}))}\\
& \geq C \left(\sum_j {k_j}\vol_\C(D_j)\right)^{1/2}\\
& = C \vol_\C(D)^{1/2}.
\end{align}

We deduce that $1/2\in\mathcal{A}(X)$. Thus the concordance is $1/2$ and it is achieved. This ends the proof of Theorem \ref{thm-surf-ab}.



\section{K3 Surfaces}


\subsection{Preliminaries}

A real K3 surface is here a real \emph{algebraic} surface $X$ such that~${{\rm H}^1(X(\C);\Z)=0}$ and the canonical divisor $K_X$ is trivial. 


\subsubsection{Exceptional curves}

On a K3 surface, a complex irreducible curve $C$ with negative self-intersection must have self-intersection $-2$, by the genus formula. In contrast, when $C$ is a real curve that is irreducible over $\R$ and has negative self-intersection, we can also have $C^2=-4$. Indeed the curve $C$ can have the form~${E+\sigma(E)}$, where $E$ is a complex $(-2)$-curve with $E\cdot\sigma(E)=0$.

By extension we call a real effective divisor \emph{exceptional} if it has self-intersection $-2$ or it has the form $E+\sigma(E)$, where $E$ is a complex curve with $E^2=-2$ and~${E\cdot\sigma(E)=0}$ (this implies $(E+\sigma(E))^2=-4$). We denote by $\Delta\subset\NS(X_\R;\Z)$ the set of classes of exceptional curves. By description of the K\"ahler cone (cf \cite[\textsection VIII (3.9)]{bpvdv}),
\begin{equation}
\Amp(X_\R)=\{\theta\in\Pos(X_\R)\,|\,\theta\cdot d>0\quad\forall d\in\Delta\}.
\end{equation}
This cone coincides with one of the chambers of $\Pos(X_\R)\backslash\bigcup_{d\in\Delta}d^\perp$. As a special case, we see that the lack of exceptional curves is equivalent to the equality~$\Amp(X_\R)=\Pos(X_\R)$.


\subsubsection{Torelli theorem} Let us recall the following result, which describes automorphisms of K3 surfaces (see {\cite[\textsection VIII (11.1) \& (11.4)]{bpvdv}}, {\cite[\textsection VIII (1.7)]{silhol}}).

\begin{theorem}[real Torelli theorem]\label{torelli}
Let $X$ be a real $K3$ surface and let $\phi$ be a Hodge isometry of ${\rm H}^2(X(\C);\Z)$ that preserves the ample cone. Then there exists a unique \emph{complex} automorphism $f$ of $X(\C)$ such that $f_*=\phi$. If moreover~$\phi$ commutes with the involution $\sigma_X^*$, then $f$ is a \emph{real} automorphism.
\end{theorem}

\begin{rema}\label{kernel}
The kernel of the representation $\Aut(X_\R)\to \Isom(\NS(X_\R;\Z))$ is finite, where $\Isom(\NS(X_\R;\Z))$ denotes the group of isometries of $\NS(X_\R;\Z)$. Indeed, the space ${\rm H}^2(X(\C);\R)$ decomposes into the orthogonal direct sum $V_1\oplus V_2 \oplus V_3$, where~$V_1=\NS(X_\R;\R)$, $V_2$ stands for the orthogonal of $V_1$ in ${\rm H}^{1,1}(X(\C);\R)$, and~$V_3=({\rm H}^{0,2}\oplus{\rm H}^{2,0})(X(\C);\R)$. The intersection form is negative definite on $V_2$ and positive definite on $V_3$. If $f\in\Aut(X_\R)$ is in the kernel of the representation, then the induced map $f_*$ on ${\rm H}^2(X(\C);\R)$ preserves the intersection form, so it is contained in the compact set $\{\id_{V_1}\}\oplus\Isom(V_2)\oplus\Isom(V_3)$. Since the matrix of $f_*$ must also have integer coefficients in a base of ${\rm H}^2(X(\C);\Z)$, there are finitely many possibilities for $f_*$, and thus for $f$ by uniqueness in the Torelli theorem.
\end{rema}


\subsection{Picard number $2$}

When $\rho(X_\R)=2$, the nef cone $\Nef(X_\R)$ has exactly two extremal rays. We say this cone is \emph{rational} if both rays are rational, that is, if they contain an element of $\NS(X_\R;\Z)\backslash\{0\}$.

\begin{theorem}\label{k3-rho=2}
Let $X$ be a real K3 surface with $\rho(X_\R)=2$.
\begin{itemize}
\item[(1)]
If the intersection form on $\NS(X_\R;\Z)$ represents $0$, or if there are exceptional curves on $X$, then the group $\Aut(X_\R)$ is finite, the cone $\Nef(X_\R)$ is rational and  $\alpha(X)=1$, the concordance being achieved.
\item[(2)]
Otherwise $X$ admits a real hyperbolic type automorphism $f$. Such an automorphism spans a finite index subgroup of $\Aut(X_\R)$, and $\alpha(X)=\frac{\h(f_\R)}{\h(f_\C)}$.
\end{itemize}
\end{theorem}

\begin{rema}
The intersection form on $\NS(X_\R;\Z)$ represents $0$ if and only if there exists some real elliptic fibration (see \cite[Corollary 3 in \textsection 3]{safarevic-torelli}).
\end{rema}

\begin{proof}
\emph{First case: the intersection form represents $0$.} The group $\Isom(\NS(X_\R;\Z))$ is finite (cf \textsection \ref{section-reseaux}), and so is the kernel of $\Aut(X_\R)\to\Isom(\NS(X_\R;\Z))$, hence $\Aut(X_\R)$ must be finite. The extremal rays of $\Nef(X_\R)$ are either isotropic half-lines, or orthogonal to the class of an exceptional curve: they are rational in both cases.

\emph{Second case: the intersection form does not represent $0$ and there are exceptional curves.} We show that $X$ has many exceptional curves. To be more precise, let~$d\in\Delta$ be the class of such a curve, and let $\phi'$ be a hyperbolic isometry of $\NS(X_\R;\Z)$. Replacing $\phi'$ by a positive iterate if necessary, the isometry $\phi'\oplus\id_{\NS(X_\R;\Z)^\perp}$ extends to an isometry $\phi$ on $\NS(X_\C;\Z)$ (cf Lemma \ref{lemme-alg-com}). Note that we cannot apply the Torelli theorem here, for $\phi$ does not preserve the ample cone (even if we suppose that it preserves $\Pos(X_\R)$). Nevertheless, we show the following lemma.

\begin{lemma}
For all $n\in\Z$, $\pm\phi^n(d)$ is the class of an exceptional curve.
\end{lemma}

\begin{proof}
If $d^2=-2$, then $\phi^n(d)^2=-2$. The Riemann--Roch formula shows that ${h^0(X,\O_X(\phi^n(d)))+h^0(X,\O_X(-\phi^n(d)))\geq 2}$, hence $\phi^n(d)$ or $-\phi^n(d)$ is the class of an effective divisor.

Otherwise, $d=e-\sigma^*e$ (the minus sign comes from (\ref{pb-sign})) with $e^2=-2$ and ${e\cdot\sigma^*e=0}$, where $e$ is the class of a complex effective divisor. From the same argument, it comes that $\phi^n(e)$ or $-\phi^n(e)$ is also the class of a complex effective divisor. Thus $\pm\phi^n(d)=\pm(\phi^n(e)-\sigma^*\phi^n(e))$ is the class of a real effective divisor.
\end{proof}

Observe that when $n$ goes to $\pm\infty$, the lines $\R\phi^n(d)$ converge to the two isotropic lines. Consequently, there are exceptional curves whose classes are arbitrarily close to both isotropic directions. The cone $\Amp(X_\R)$ is one of the chambers of $\Pos(X_\R)\backslash\bigcup_{d\in\Delta}d^\perp$, so both its extremal rays must be orthogonal to classes of exceptional curves, hence they are rational. Since the subgroup of $\Isom(\NS(X_\R;\Z))$ whose elements fix or exchange these extremal rays is finite, the group $\Aut(X_\R)$ is also finite by Remark \ref{kernel}.

In the first two cases, let $\R^+[D]$ be an extremal ray of the cone $\Nef(X_\R)$. By Riemann--Roch, $h^0(X,\O_X(D))\geq 2$, thus $\mvol_\R(D)>0$ by Proposition \ref{pinceau}. The assertion about $\alpha(X)$ follows, using Proposition \ref{conenef}.

\emph{Third case: the intersection form does not represent $0$ and there is no exceptional curve.} Let $\phi'$ be a hyperbolic isometry of $\NS(X_\R;\Z)$ that preserves the cone~$\Amp(X_\R)=\Pos(X_\R)$. By the same argument as in the proof of Theorem~\ref{delta}, some iterate $\phi'^k$ extends to a Hodge isometry $\phi$ of ${\rm H}^2(X(\C);\Z)$ that commutes with~$\sigma^*$. Then by the Torelli theorem, there exists on $X$ a real automorphism $f$ such that $f_*=\phi$. The entropy of $f_\C$ is positive as a multiple of the spectral logradius of $\phi'$. Now the representation $\Aut(X_\R)\to\Isom(\NS(X_\R;\Z))$ has finite kernel, and the subgroup generated by $f_*$ has finite index in $\Isom(\NS(X_\R;\Z))$~(cf~\textsection\ref{section-reseaux}). It follows that $f$ spans a finite index subgroup of $\Aut(X_\R)$. The concordance formula is a consequence of Theorem \ref{formule-exacte}.
\end{proof}

\begin{example}[\cite{wehler}]\label{wehler}
Let $Y$ be the $3$-dimensional flag variety
\begin{equation}
Y=\{(P,L)\in\P^2(\C)\times{\P^2(\C)}^*\,|\,P\in L\}.
\end{equation}
Let $X$ be a smooth hypersurface of $Y$ such that the projections $\pi_1:X(\C)\to\P^2(\C)$ and $\pi_2:X(\C)\to{\P^2(\C)}^*$ are ramified $2$-coverings. Then $X$ is a K3 surface, called~a Wehler surface. Furthermore, generic Wehler surfaces have a rank $2$ N\'eron--Severi group, spanned by generic fibers of the two coverings. The automorphism group is then isomorphic to a free product $\Z/2\Z * \Z/2\Z$, the generators being the involutions~$s_1$ and $s_2$ of the coverings $\pi_1$ and $\pi_2$, respectively. Furthermore, the automorphism~$f=s_1\circ s_2$ has hyperbolic type.

If moreover the surface $X$ is defined over $\R$, then ${\NS(X_\R;\Z)=\NS(X_\C;\Z)}$, and the automorphism $f$ is real. So, by Theorem \ref{formule-exacte},
\begin{equation}
\alpha(X)=\frac{\h(f_\R)}{\h(f_\C)}=\frac{\h(f_\R)}{\log(7+4\sqrt{3})}.
\end{equation}
\end{example}


\subsection{Deformation of K3 surfaces in $\P^1\times\P^1\times\P^1$}\label{deformation}

The following example was first described by McMullen \cite{mcmullen-k3-siegel}. Fix a nonzero real number $t$. Let $X^t$ be the hypersurface of $\P^1(\C)^3$ defined in its affine part $\C^3$ by
\begin{equation}
(z_1^2+1)(z_2^2+1)(z_3^2+1)+tz_1z_2z_3=2.
\end{equation}
It is a smooth surface of tridegree $(2,2,2)$, hence a K3 surface \cite{mazur}, here defined over~$\R$. We have three double (ramified) coverings ${\pi_j^t:X^t\to\P^1\times\P^1}$ (with ${j\in\{1,2,3\}}$) that consist in forgetting the $j$-th coordinate. These three coverings give rise to three involutions $s_j^t$ on $X^t$ that span a subgroup of $\Aut(X_\R)$, which is a free product $\Z/2\Z*\Z/2\Z*\Z/2\Z$ \cite{wang}. Let $f^t$ be the automorphism of $X^t$ obtained by composing these three involutions. Its entropy can be computed by the action on the subgroup of $\NS(X_\R;\Z)$ spanned by the fibers of the three coverings. We obtain $\h(f^t_\C)=\log(9+4\sqrt{5})$ (see, for instance, \cite{cantat-panorama}).

For parameter $t=0$, the complex surface $X^0(\C)$ is not smooth, for there are $12$ singular points $(\infty,\pm i,\pm i)$, $(\pm i,\infty,\pm i)$, and $(\pm i,\pm i,\infty)$. However, these points are not real, so the surface $X^t(\R)$ remains smooth at $t=0$. Restricted to~$X^0(\R)$, the birational map $f^0$  is an order $2$ diffeomorphism given by the formula $f^0(x_1,x_2,x_3)=(-x_1,-x_2,-x_3)$. Consequently  $\h(f^0_\R)=0$.

Let $\mathcal{X}$ be the submanifold of $\P^1(\R)^3\times\R$ defined by $\mathcal{X}=\{(x,t)\,|\,x\in X^t(\R)\}$. The projection $p:\mathcal{X}\to\R$ is a locally trivial bundle whose fibers are the real surfaces~$X^t(\R)$. Thus there is an open neighborhood $I_\epsilon=(-\epsilon,\epsilon)$ around $0$, and an injective local diffeomorphism $\psi:X^0(\R)\times I_\epsilon\to \mathcal{X}$ such that $p\circ\psi$ is the natural projection on the second coordinate. For all $t\in I_\epsilon$ the map $\psi$ induces a diffeomorphism from $X^0(\R)$ to $X^t(\R)$, which enables us to conjugate ${f^t_\R:X^t(\R)\to X^t(\R)}$ to a diffeomorphism $g^t$ on $X^0(\R)$. This family of diffeomorphisms on $X^0(\R)$ is a continuous family for the $C^\infty$-topology. As the map $g^0=f^0_\R$ has entropy $0$, it follows that ${\lim_{t\to 0}\h(g^t)=0}$, by continuity of the topological entropy on~${\rm Diff}^\infty(X^0(\R))$ (see \cite{yomdin} or \cite{newhouse-continuity} for the upper semicontinuity, and \cite[Corollary S.5.13]{katok-hasselblatt} for the lower semicontinuity). Since the entropy does not change by conjugacy, we also have ${\lim_{t \to 0}\h(f^t_\R)=0}$. On the other hand we obtain, by Theorem~\ref{entropie}, that~${\alpha(X^t)\leq \h(f^t_\R)/\h(f^t_\C)=\h(f^t_\R)/\log(9+4\sqrt{5})}$ for $t\neq 0$, and so we get
\begin{equation}
\lim_{t\to 0,\,t\neq 0}\alpha(X^t)=0.
\end{equation}

To sum up, we have found a family $(X^t,f^t)_{t\in I_\epsilon\backslash\{0\}}$ of real K3 surfaces $X^t$ embedded in $(\P^1)^3$, equipped with a real hyperbolic type automorphism $f^t$, such that
\begin{itemize}
\item[(1)]
$\h(f^t_\C)$ is a positive constant;
\item[(2)]
as $t$ goes to $0$, $X^t(\R)$ degenerates in a smooth surface, and $f^t_\R$  in a $0$-entropy diffeomorphism;
\item[(3)]
$\lim_{t\to 0}\alpha(X^t)=0$.
\end{itemize}

So we have just proved the following theorem.

\begin{theorem}\label{alpha-petit}
For any $\eta>0$ there exists a real K3 surface in $(\P^1)^3$ such that $\alpha(X)<\eta$.
\end{theorem}



\section{Nondensity of Automorphisms in $\Diff(X(\R))$}\label{section-non-dense}

Let $X$ be a real algebraic surface. For any $r\in\N\cup\{\infty\}$ we denote by~$\Diff^r(X(\R))$ the group of $C^r$-diffeomorphisms of the surface $X(\R)$, together with its $C^r$-topology (when $r=0$, $\Diff^0(X(\R))=\Homeo(X(\R))$ stands for homeomorphisms). The group $\Aut(X_\R)$ identifies with a subgroup of $\Diff^r(X(\R))$. We would like to know how this subgroup sits into the whole group of diffeomorphisms.

When $\Aut(X_\R)$ does not have any positive entropy element, it obviously cannot be dense in $\Diff^\infty(X(\R))$. Indeed, there always exist positive entropy diffeomorphisms on $X(\R)$, and these diffeomorphisms cannot be approached by any automorphism, by the continuity of the entropy. The following result reverses this idea.

\begin{proposition}\label{non-dense}
Let $X$ be a real algebraic surface such that $\alpha(X)>0$. Then the image of the group $\Aut(X_\R)$ in $\Diff^\infty(X(\R))$ is not dense.
\end{proposition}

\begin{proof}
Let $g$ be a diffeomorphism of $X(\R)$ such that $0<\h(g)<\alpha(X)\log(\lambda_{10})$. Since the entropy varies continuously on $\Diff^\infty(X(\R))$, there exists a neighborhood~$U$ of $g$ in $\Diff^\infty(X(\R))$ such that the entropy of any diffeomorphism in $U$ remains in the open interval $(0,\alpha(X)\log(\lambda_{10}))$. By Corollary \ref{lehmer} this neighborhood cannot contain any automorphism of $X$.
\end{proof}

In \cite{kollar-mangolte}, Koll\'ar and Mangolte established the nondensity of the image of $\Aut(X_\R)$ in $\Homeo(X(\R))$ as soon as $X(\R)$ has the topology of a connected orientable surface with genus $\geq 2$. In contrast, they proved, for surfaces birational to $\P^2_\R$, the density in $\Diff^\infty(X(\R))$ of the group of \emph{birational} transformations with imaginary indeterminacy points.

When the Kodaira dimension is $0$, $X(\R)$ is naturally equipped with a canonical volume form $\mu_X$, which comes from an everywhere nonzero holomorphic $2$-form on some finite cover of $X$. Since each automorphism preserves $\mu_X$, the nondensity is obvious, as pointed out in Koll\'ar and Mangolte \cite{kollar-mangolte}.  Nevertheless, we can prove, using exactly the same argument, the nondensity in diffeomorphisms that preserve the volume.

\begin{propbis}
Let $X$ be a real algebraic surface of Kodaira dimension $0$ such that $\alpha(X)>0$. Then the image of the group $\Aut(X_\R)$ in $\Diff^\infty_{\mu_X}(X(\R))$ is not dense, where $\Diff_{\mu_X}^\infty(X(\R))$ denotes the subgroup of $\Diff^\infty(X(\R))$ whose elements preserve the canonical volume form $\mu_{X}$.
\end{propbis}

Actually we can be more precise when the automorphism group $\Aut(X_{\C})$ is a discrete group (for the uniform convergence topology), that is, when the connected component $\Aut(X_{\C})_{0}$ of the identity is reduced to a single point. For instance, this is the case for K3 and Enriques surfaces, but not for tori (for which $\Aut(X_{\C})_{0}$ consists in all translations).

\begin{theorem}
Let $X$ be a real algebraic surface. Assume that $\alpha(X)>0$ and ${\Aut(X_{\C})_0=\{\id_X\}}$. Then the image of the group $\Aut(X_{\R})$ in $\Diff^1(X(\R))$ is a discrete subgroup.
\end{theorem}

\begin{proof}
Fix $\alpha>0$ such that $\alpha\in\mathcal{A}(X)$. There are positive numbers $q$ and $C$ such that any real ample divisor $D$ whose class is $q$-divisible satisfies 
\begin{equation}\label{truc}
\mvol_\R(D)\geq C\vol_\C(D)^\alpha.
\end{equation}
Let $D_0$ be such a divisor and let $M>1$ be such that  $CM^\alpha>\mvol_\R(D_0)$ (in particular, $\vol_\C(D_0)<M$).

\begin{lemma}\label{lemm-discret}
The set $\Gamma=\{f\in\Aut(X_\R)\,|\,\vol_\C(f_*D_0)\leq M\}$ is finite.
\end{lemma}

\begin{proof}[Proof of Lemma \ref{lemm-discret}]
Denote by $\Theta\subset\NS(X_\R;\Z)$ the set of classes of ample divisors~$D$ that satisfy $\vol_\C(D)\leq M$. This set is finite because such classes are in the compact set $\{\theta\in\Nef(X_\R)\,|\,\theta\cdot[\kappa]\leq M\}$, where $\kappa$ denotes the K\"ahler form on $X$.

By \cite[2.2]{lieberman} or \cite[4.8]{fujiki} the subgroup $\{f\in\Aut(X_\R)\,|\,f_*[D_0]=[D_0]\}$ has finitely many connected components, so in our case it is finite. It follows that the set $\Gamma=\{f\in\Aut(X_\R)\,|\,f_*[D_0]\in\Theta\}$ is finite.
\end{proof}

As $\Gamma$ is finite and $\frac{CM^\alpha}{\mvol_\R(D_0)}>1=\|{\rm D}\id_{X(\R)}\|_\infty$, we can find a neighborhood $U$ of $\id_{X(\R)}$ in $\Diff^1(X(\R))$ such that
\begin{itemize}
\item[(1)] $U\cap\Gamma=\{\id_{X(\R)}\}$;
\item[(2)] for all $g\in U$, $\|{\rm D}g\|_\infty<\frac{CM^\alpha}{\mvol_\R(D_0)}$.
\end{itemize}
Let $f$ be a real automorphism of $X$ such that the restricted map $f_\R:X(\R)\to X(\R)$ is in $U$. Since the length of a curve is at most multiplied by $\|{\rm D}f_\R\|_\infty$ when we take its image by $f_\R$, we obtain
\begin{equation}
\mvol_\R(f_*D_0) \leq \|{\rm D}f_\R\|_\infty \mvol_\R(D_0) < CM^\alpha.
\end{equation}
On the other hand, $\mvol_\R(f_*D_0)\geq C\vol_\C(f_*D_0)^\alpha$, so we get $\vol_\C(f_*D_0)<M$, hence $f\in\Gamma$. Then by hypothesis on $U$, we get $f_\R=\id_{X(\R)}$. This implies that~$\Aut(X_\R)$ is a discrete subgroup of $\Diff^1(X(\R))$.
\end{proof}



\appendix

\section{Cauchy--Crofton Formula and Consequences}\label{annexe-crofton}

What is described here can be found in the manuscript \cite{christol}, except for Lemma \ref{deg-2}, the proof of which is incomplete in Christol \cite{christol}.

There is a classical way, in integral geometry (see \cite{santalo}), to compute the volume of a $k$-dimensional submanifold $N$ of $\P^d(\R)$, just by taking the mean of the number of intersections between $N$ and $k$-codimensional projective subspaces. In order to make it work, we choose both a metric on $\P^d(\R)$, and a probability measure on the Grassmannian $\mathbf{G}(d-k,d)$ (i.e., the real algebraic variety of $(d-k)$-dimensional projective subspaces of $\P^d(\R)$), which are invariant under the action of the orthogonal group ${\rm O}(d+1)$. Namely, we set the Fubini--Study metric on $\P^d(\R)$, and the probability $\mu_{d-k,d}$ on $\mathbf{G}(d-k,d)$ induced by the Haar measure on ${\rm O}(d+1)$ (the Grassmannian is homogeneous with respect to this group). Now we can state the formula.

\begin{theorem}[Cauchy--Crofton formula]
Let $N$ be a $k$-dimensional submanifold of $\P^d(\R)$. With respect to the Fubini--Study metric,
\begin{equation}\label{formule-crofton}
\vol(N)=\vol(\P^k(\R))\int_{\Pi\in\mathbf{G}(d-k,d)}\sharp(N\cap \Pi)\,{\rm d}\mu_{d-k,d}(\Pi).
\end{equation}
\end{theorem}

It is enough to check the formula when $N$ is a $k$-simplex, and then to approach an arbitrary submanifold by such simplices. A similar formula in the Euclidean context can be found in \cite[p. 245 (14.70)]{santalo}.

\begin{corollary}\label{1.11}
Let $Y$ be a real $k$-dimensional algebraic subvariety of $\P^d_{\R}$. With respect to the Fubini--Study metric, 
\begin{equation}\label{vol-reel}
\vol_\R(Y)\leq\deg(Y)\vol_\R(\P^k),
\end{equation}
with equality if and only if $Y$ is a union of $\deg(Y)$ real projective subspaces.
\end{corollary}

\begin{proof}[Proof of Corollary \ref{1.11}]
Observe that $\sharp(Y(\R)\cap\Pi)\leq\deg(Y)$ for all $\Pi\in\mathbf{G}(d-k,d)$. Thus we get Inequality (\ref{vol-reel})  using the Cauchy--Crofton formula.

The equality is obviously achieved when $Y$ is the union of $\deg(Y)$ projective subspaces. Now we prove that this condition is necessary.

\begin{lemma}\label{deg-2}
Let $Y$ be a geometrically irreducible real $k$-dimensional algebraic subvariety of $\P^d_{\R}$. If $\deg(Y)>1$, then there exists a real projective $k$-codimensional subspace $\Pi$ such that the number of real points of $Y\cap\Pi$, counted with multiplicities, is no more than $\deg(Y)-2$.
\end{lemma}

\begin{proof}
Observe that the assumptions imply $0<k<n$. By Bertini's theorem (see \cite[3.3.1]{lazarsfeld}), there exists a real projective subspace $L$ of dimension ${d-k+1\geq 2}$, such that the curve $C=Y\cap L$ is irreducible over $\C$ and $\deg(C)=\deg(Y)$.

First we suppose that there is no hyperplane of $L$ containing the curve $C$. We choose two distinct complex conjugate points $P$ and $\b P$ on the curve $C(\C)$, and a real hyperplane $\Pi$ of $L$ such that $\Pi(\C)$ contains these two points. As $C$ is irreducible and not contained in $\Pi$, the intersection $C\cap\Pi=Y\cap \Pi$ is a finite number of points, including the complex points $P$ and $\b P$. The number of complex points of this intersection, counted with multiplicities, is exactly $\deg(Y)$, and at least two complex points are not real. The result follows.

Otherwise, let $L'\varsubsetneq L$ be the minimal projective subspace that contains the curve~$C$. As $\deg(C)>1$, it follows that $\dim(L')\geq 2$. By the first step, we can choose a hyperplane $\Pi'\subset L'$ such that $\sharp(C\cap\Pi')(\R)\leq\delta-2$. Then we take any hyperplane~$\Pi$ of $L$ containing $\Pi'$ and not $L'$, and we are done.
\end{proof}

Let us go back to the proof of the case of equality. Let $Y$ be a real $k$-dimensional subvariety of $\P^d$ that is not the union of real projective subspaces. We may assume that $Y$ is irreducible over $\R$. If it is not geometrically irreducible, then~${Y=Z\cup\sigma(Z)}$, with $Z$ a complex subvariety that is not real, hence ${\vol_\R(Y)=0 <\deg(Y)\vol_\R(\P^k)}$. Otherwise we can deduce from Lemma \ref{deg-2} that there exists a real hyperplane $\Pi$ such that $\sharp (Y\cap\Pi)(\R)\leq\deg(Y)-2$ (with multiplicities). This inequality remains satisfied on a neighborhood of $\Pi$ in the Grassmannian ${\mathbf{G}(d-k,d)}$. Such a neighborhood has a positive probability for $\mu_{d-k,k}$, so the Cauchy--Crofton formula implies~${\vol_\R(Y)<\deg(Y)\vol_\R(\P^k)}$.
\end{proof}



\bibliographystyle{abbrv}
\bibliography{biblio}

\end{document}